\documentclass[10pt,a4paper]{article}

\usepackage[utf8]{inputenc}
\usepackage[margin=1.5cm]{geometry}
\usepackage[sorting=nyt,sortcites=true,citestyle=numeric-comp]{biblatex}
\usepackage{mathtools,amssymb, amsthm, amsfonts, mathrsfs, braket,tikz-cd,accents}
\usepackage[sans]{dsfont}
\usepackage[bbgreekl]{mathbbol}
\usepackage[english]{isodate}
\usepackage[hidelinks]{hyperref}
\usepackage[shortlabels, inline]{enumitem}
\usepackage{epigraph}
\usepackage{titlesec}
\usepackage{authblk}

\title{Ollivier Curvature Bounds for the Brownian Continuum Random Tree}
\author{Christy Kelly\footnote{\href{mailto:christy.kelly@riken.jp}{christy.kelly@riken.jp}}}
\affil{RIKEN iTHEMS, Wako, Saitama 351-0198}
\date{\printdayoff\today}

\setlist[enumerate,1]{label={(\roman *)}}

\theoremstyle{plain}
\newtheorem{proposition}{Proposition}[section]
\newtheorem{theorem}[proposition]{Theorem}
\newtheorem{lemma}[proposition]{Lemma}
\newtheorem{corollary}[proposition]{Corollary}

\newtheorem{fact}[proposition]{Fact}

\theoremstyle{definition}
\newtheorem{definition}[proposition]{Definition}

\theoremstyle{remark}

\newtheorem*{remarks}{Remark}

\renewcommand{\d}{\text{d}}

\begin{document}
	
	\maketitle
	\begin{abstract}
		We compute bounds in the expected Ollivier curvature for the Brownian continuum random tree $\mathcal{T}_{\mathbb{e}}$. The results indicate that when the scale dependence of the Ollivier curvature is properly taken into account, the Ollivier-Ricci curvature of $\mathcal{T}_{\mathbb{e}}$ is bounded above by every element of $\mathbb{R}$ for almost all points of $\mathcal{T}_{\mathbb{e}}$. This parallels the well-known result that every continuum tree is a $CAT(K)$ space for all $K\in\mathbb{R}$.
	\end{abstract}
	\section{Introduction}
    A \textit{real tree} $\mathcal{T}$ is a compact metric space that behaves like a graph-theoretic tree in terms of its path-structure: there is a unique curve without self-intersections between any two points of $\mathcal{T}$ with length equal to the distance between the two points. Despite the apparent simplicity of their structure, real trees have important topological applications: c.f. e.g. Ref. \cite{Bestvina-RealTrees}. They also have a surprisingly rich (metric) geometry. The key feature from this perspective is the fact that every triangle in a real tree is `infinitely thin'; more precisely every triangle in a continuum tree looks like some imbedding of the claw graph $K_{1,3}$. As such it is intuitively clear that real trees are in some sense infinitely hyperbolic spaces. In fact, it is well known that a geodesic space is a real tree iff it is $CAT(K)$ for all $K\in \mathbb{R}$ \cite{BridsonHaefliger-MetricSpacesNonPositiveCurvature,AndreevBerestovskii-DimensionRealTrees}. 

    From the $CAT(K)$ property one immediately sees that a real tree is simply connected, while an elementary argument shows that the topological dimension of every real tree is one; there seems to be little else to restrict the global geometry of real trees in general: explicit examples of real trees with arbitrarily large (including infinite) Hausdorff dimension can be given \cite{AndreevBerestovskii-DimensionRealTrees}. There are, however, classes of \textit{random} real trees which have slightly better global properties. Perhaps the best understood is the Brownian continuum random tree (BCRT) $\mathcal{T}_{\mathbb{e}}$ of Aldous \cite{Aldous-CRTI,Aldous-CRTII,Aldous-CRTIII,LeGall-RandomTreesApplications}, a random real tree that arises as the universal scaling limit of a variety of random discrete tree-valued processes, including especially the genealogy trees associated to (critical) branching processes \cite{Aldous-CRTI,Aldous-CRTIII,MarckertMiermont-CRTScalingLimitUnorderedBinaryTrees,Stufler-CRTScalingLimitUnlabelledUnrootedTrees,HaasMiermont-ScalingLimit,Miermont-InvariancePrincipleSpatialMultitype,BroutinMarckert-AsymptoticsTreesPrescribedDegreeSequence,CurienHaasKortchemski-CRTScalingLimitRandomDissections}. $\mathcal{T}_{\mathbb{e}}$ also plays an important role in constructions of several other probabilistic structures such as the scaling limit of the metric space of connected components of a critical Erd\H{o}s-R\'{e}nyi graph \cite{AddarioBerryBroutinGoldschmidt-ContinuumLimitCriticalRandomGraphs,AddarioBerryBroutinGoldschmidt-CriticalRandomGraphs} and the Brownian map \cite{MarckertMokkadem-LimitQuadrangulations,LeGall-UniquenessUniversalityBM,LeGall-GeodesicsLargePlanarMaps,LeGall-TopologicalStructureBM,Miermont-BMScalingLimitPlaneQuandrangulations,DuplantierMillerSheffield-LQGMatingTrees}. Closer to the author's interests, the BCRT appears as the scaling limit of tensor models \cite{ADJ-QG,GurauRyan-MelonsBranchedPolymers}, where it represents a pathological model of quantum spacetime. Note that the latter two examples indicate that the BCRT appears naturally in the context of studies of quantum gravity in physics where it is referred to as a \textit{branched polymer}.

    We can again consider the (metric-measure) geometry of the BCRT; its global geometry is well understood: the Hausdorff dimension of the $\mathcal{T}_{\mathbb{e}}$ is $2$ and its spectral dimension---a quantity controlling the dominant scaling of the heat kernel of a Brownian motion on $\mathcal{T}_{\mathbb{e}}$---is known to be $4/3$ \cite{ADJ-QG,DuquesneLeGall-ProbabilisticFractalAspectsLevyTrees,HaasMiermont-Genealogy,Kigami-HarmonicCalculusLimitsNetworks}. At a local level, Duquesne and Le Gall further obtain information about the Hausdorff measure of the BCRT (more generally of the so-called \textit{stable L\'{e}vy trees}) which include local almost sure bounds on the volume growth of balls \cite{DuquesneLeGall-HausdorffMeasure}; in particular the results of Duquesene and Le Gall establish that locally one has almost sure fluctuations of order $\log(\log(\delta^{-1}))$ around the dominant $\delta^2$ scaling. Croydon has extended the analysis of the volume growth problem in the context of the BCRT in Ref. \cite{Croydon-VolumeGrowth}, obtaining slightly more explicit information about the local volume growth bounds (improved constants, though the order of the fluctuations agrees with the result of Duquesne and Le Gall) as well as global bounds on the rate of volume growth for balls. Following the work of Kigami \cite{Kigami-HarmonicCalculusLimitsNetworks} and his own work \cite{Croydon-HeatKernelFluctuations} on heat kernel fluctuations on metric-measure spaces with a resistance form, Croydon succeeded in constructing a Brownian motion on $\mathcal{T}_{\mathbb{e}}$ and obtaining heat kernel bounds via the global bounds on the volume fluctuations. These results readily imply the dimension results listed previously.

    Another set of local quantities of potential interest are the recently introduced \textit{synthetic curvatures} of optimal transport theory c.f. Ref. \cite{Villani-OptimalTransport} for a comprehensive review of the field. The optimal transport problem is briefly concerned with the minimisation of the total cost of transforming one probability distribution into another for given cost functions; it turns out that when the cost is given by some power $p\in [1,\infty)$ of the distance in a metric space, the $p$th root of the transport cost gives the \textit{Wasserstein $p$-distance} $\mathcal{W}_p$, a metric on (a suitable restriction of) the space of Borel probability measures of the original space. For the case $p=1$, Ollivier introduced a synthetic curvature which after suitable normalisation approximates the Ricci curvature in Riemannian manifolds \cite{Ollivier-RCMCMS,Ollivier-RCMS}. The Ollivier curvature has several nice properties: it is very intuitive (spaces are positively curved because the average distance between small balls is less then the distance between their centres) and is well-defined and relatively computable for discrete systems. This has made it a popular quantity to study in network geometry c.f. e.g. Refs. \cite{Ni_RicIntTop, Sandhu_Cancer, Sandhu_Market, TannenbaumEtAl_Cancer, WangEtAl_DiffGeom, WangEtAl_Interference, WangEtAl_QUBO, WangEtAl_WirelessNetwork, WhiddenMatsen_SubtreeGraph, FarooqEtAl_Brain, SiaEtAl_CommunityDetection, JostLiu_RicciCurv,BogunaEtAl-NetworkGeometry} as well as in some physical models of quantum gravity \cite{KlitgaardLoll_HowRound,KlitgaardLoll_ImplementingQuantumRicciCurvature,KlitgaardLoll_QuantumRicciCurvature,BrunekreefLoll-CurvatureProfile,Gorard,Trugenberger_CombQG,KellyEtAl,KellyEtAl_Circle}. The Ollivier-Ricci curvature is also rather explicit being a precise \textit{characterisation} of the Ricci curvature, rather than just a bound. Despite this, it has perhaps garnered less interest in less applied sectors. Part of the issue is perhaps the relative subtlety of its stability theory (see \cite{Hoorn-Convergence,HoornEtAl_OllCurvConv,KellyTrugenbergerBiancalana-ConvergenceCombinatorialGravity}), but more importantly we should look to the richness of optimal transport theory in the $p=2$ case \cite{Ollivier-RCMCMS}. In particular, Sturm \cite{Sturm-GeomI} and Lott and Villani \cite{LottVillani-SyntheticCurvature} independently defined a curvature-dimension condition when $p=2$ that has since perhaps become the canonical example of a synthetic curvature for metric-measure spaces. In the context of continuum trees, however, the Sturm-Lott-Villani curvature-dimension condition is perhaps not appropriate: curvature-dimension conditions essentially specify \textit{lower bounds} on the Ricci curvature which clearly cannot be done when the space in question is infinitely hyperbolic.

    The main task in computing the Ollivier curvature in a metric measure space $(X,\rho_X,\mu)$ is an evaluation of the $\mathcal{W}_1$-distance between two suitable Borel probability measures in $X$, each typically taken to be the uniform probability measures induced by $\mu$ on open balls of a given radius and specified centre. Associated, then, to the Ollivier curvature are two free parameters $\delta$ and $\ell$ taking (sufficiently small) values in $(0,\infty)$; these describe the radius of the balls in question and the distance between their centres respectively. In the smooth context, the Ollivier curvature agrees with the Ricci curvature \textit{in the asymptotic regime where $\delta$ and $\ell$ are small and only after a suitable rescaling with respect to $\delta$}. Thus it is natural to consider a scale-free analogue of the Ollivier curvature by taking the limit $\delta,\:\ell\rightarrow 0$ of the naive Ollivier curvature following rescaling. A more complete discussion of these points as well as a more formal account of the Ollivier curvature is given in section \ref{section: OllivierCurvature}. 

    In section \ref{section: BCRT} we present some basic material on continuum trees in general and the BCRT in particular. Most of this material is well known, and can be skipped (up to notational peculiarities) by a reader familiar with the BCRT. A possible exception is proposition \ref{proposition: BallIntersection} which provides a nice decomposition of the intersection of two open balls in arbitrary continuum trees. 

    The main result of this paper is an explicit computation of bounds for the expected Ollivier curvature in the BCRT, as presented in corollary \ref{corollary: OllCurvBound}. This follows immediately from a computation of bounds for the $L_1$-Wasserstein distance between suitable probability measures in $\mathcal{T}_{\mathbb{e}}$ as given in theorem \ref{theorem: MainTheorem}. Essentially we find that for almost all points $x\in \mathcal{T}_{\mathbb{e}}$ the Ollivier curvature $\kappa_x(\delta,\ell)$ at $x$ at the given scales $(\delta,\ell)$ satisfies the bounds
    \begin{align}
        -a\frac{\delta}{\ell}\leq \kappa_x(\delta,\ell)\leq -b\frac{\delta}{\ell}
    \end{align}
    for suitable positive constants $a,\:b\in (0,\infty)$. Note that we have been very rough in our use of notation in the above to emphasise the key point. To obtain the scale-free curvature we need to normalise by a factor of $\delta^2$ so it is easily seen that in any limit $\delta,\:\ell\rightarrow 0$ the scale-free Ollivier curvature at $x$ blows up to negative infinity. In this sense, the scale-free Ollivier curvature and the synthetic sectional curvature are compatible in this context. This result has possible ramifications for the proper formulation of (Euclidean) quantum gravity in dimensions $D\geq 3$.
	\section{Ollivier Curvature}\label{section: OllivierCurvature}
	We briefly review the fundamental properties of the Ollivier curvature we shall require in the subsequent, beginning some elementary ideas in optimal transport theory; see \cite{Villani-OptimalTransport,Ollivier-RCMCMS,Ollivier-RCMS} for more complete introductions to the relevant material. Throughout this section we work in a complete separable metric space $(X,\rho_X)$; the open ball of radius $\varepsilon>0$ centred at a point $x\in X$ is denoted $\mathbb{B}^X_\varepsilon(x)$. Also for any measurable spaces $(\Omega_1,\Sigma_1)$ and $(\Omega_2,\Sigma_2)$, any measure $\mu$ on $(\Omega_1,\Sigma_1)$ and any measurable mapping $f:\Omega_1\rightarrow\Omega_2$, we let $f_*\mu$ denote the pushforwards of $\mu$ with respect to $f$. Also, we let $\mu(f)$ denote the integral of a measurable function $f$ with respect to the measure $\mu$.
    \begin{definition}
        Let $\mu$ and $\nu$ be Borel probability measures in $X$. 
        \begin{enumerate}
            \item A \textit{transport plan} between $\mu$ and $\nu$ is a Borel probability measure $\xi$ on $X\times X$ such that
            \begin{align}\label{equation: MarginalConstraints}
                (\pi_1)_*\xi=\mu && (\pi_2)_*\xi=\nu
            \end{align}
            where $\pi_1:(x,y)\mapsto x$ and $\pi_2:(x,y)\mapsto y$ are the natural projections onto the first and second elements respectively. We refer to the conditions \ref{equation: MarginalConstraints} as the \textit{marginal constraints} satisfied by transport plans. The set of all transport plans between $\mu$ and $\nu$ is denoted $\Pi(\mu,\nu)$.
            \item The \textit{transport cost} associated to a transport plan $\xi\in \Pi(\mu,\nu)$ is defined
            \begin{align}
                \mathcal{W}_X(\xi)\coloneqq \xi(\rho_X).
            \end{align}
            The \textit{optimal transport cost} or \textit{Wasserstein distance} is then defined
            \begin{align}
                \mathcal{W}_X(\mu,\nu)\coloneqq \inf_{\xi\in \Pi(\mu,\nu)}\mathcal{W}_X(\xi).
            \end{align}
            \item A transport plan $\xi\in \Pi(\mu,\nu)$ is said to be \textit{optimal} iff $\mathcal{W}_X(\xi)=\mathcal{W}_X(\mu,\nu)$.
        \end{enumerate}
    \end{definition}
    Roughly speaking, the idea is that we have some unit quantity of substance distributed according to $\mu$ that we wish to redistribute according to $\nu$; the problem is then to find the cheapest way to conduct this redistribution if it costs $\rho_X(x,y)$ to transport a unit of substance from $x$ to $y$. Given a transport plan $\xi\in \Pi(\mu,\nu)$ the quantity $\xi(E_1\times E_2)$ indicates the proportion of substance in $E_1$ that is to be transported to $E_2$ where $E_1$ and $E_2$ are measurable subsets of $X$.
    
    In the present setting it is fairly simple to see that the mapping $\xi\mapsto \mathcal{W}_X(\xi)$ is lower semicontinuous while an elementary application of the Prokhorov theorem suffices to show that $\Pi(\mu,\nu)$ is compact with respect to the topology of weak convergence of measure. From these considerations the existence of optimal transport plans follows. $\mathcal{W}_X$ turns out to be a metric on the space of Borel probability measures on $X$; in fact we have the following:
    \begin{fact}
        Let $\mathcal{P}(X)$ denote the space of Borel probability measures on $X$ with finite first moments, i.e. for each $\mu\in \mathcal{P}(X)$ we have
        \begin{align}
            \mu(x\mapsto \rho_X(x_0,x))<\infty
        \end{align}
        for all $x_0\in X$. Then $(\mathcal{P}(X),\mathcal{W}_X)$ is a metric space and its topology is equal to the topology generated by demanding weak convergence and convergence of first central moments. $\qed$
    \end{fact}
    One remarkable feature of optimal transport theory is the existence of a dual formulation:
    \begin{definition}
        Let $\mu,\:\nu\in \mathcal{P}(X)$. For any $f:X\rightarrow [-\infty,\infty]$ that is integrable for at least one of $\mu$ and $\nu$ we define
        \begin{align}
            \mathcal{K}_{\mu,\nu}^X(f)\coloneqq \mu(f)-\nu(f)\nonumber.
        \end{align}
        The \textit{Kantorovitch-Rubinstein distance} between $\mu$ and $\nu$ is then defined
        \begin{align}
            \mathcal{K}_X(\mu,\nu)=\sup_{f\in \mathscr{L}(X)}\mathcal{K}_{\mu,\nu}(f)
        \end{align}
        where $\mathscr{L}(X)$ denotes the set of all $1$-Lipshitz maps $f:X\rightarrow \mathbb{R}$. 
    \end{definition}
    In the present setting the Kantorovitch duality theorem thus states:
    \begin{fact}
        $\mathcal{K}_X(\mu,\nu)=\mathcal{W}_X(\mu,\nu)$ for all $\mu,\:\nu\in \mathcal{P}(X)$. $\qed$
    \end{fact}
    The key point for our purposes is the following:
    \begin{corollary}
        Let $\mu,\:\nu\in \mathcal{P}(X)$. Then
        \begin{align}
            \vert \mathcal{K}_{\mu,\nu}(f)\vert\leq \mathcal{W}_X(\mu,\nu)\leq \mathcal{W}_X(\xi)
        \end{align}
         for all $f:X\rightarrow \mathbb{R}$ $1$-Lipschitz and $\xi\in \Pi(\mu,\nu)$. $\qed$
    \end{corollary}
    We now introduce the Ollivier curvature:
    \begin{definition}
        Let $(X,\rho_X,\mu)$ be a triple such that $(X,\rho_X)$ is a geodesic space and $\mu$ a $\sigma$-finite measure on $X$; for convenience we shall call any such triple a \textit{metric-measure space}.
        \begin{enumerate}
            \item For any $\delta>0$ and any $x\in X$ we define the measure $\mu_x^\delta$ via
            \begin{align}
                \mu_x^\delta(E)=\frac{\mu(E\cap \mathbb{B}^X_\delta(x))}{\mu(\mathbb{B}_\delta^X(x))}
            \end{align}
            for all measurable $E\subseteq X$.
            \item For any sufficiently small $\delta>0$ and $\ell>0$, the \textit{Ollivier curvature at scale $(\delta,\ell)$} is defined
            \begin{align}
                \kappa^{\delta,\ell}_x(\gamma)=1-\frac{\mathcal{W}_X(\mu^\delta_x,\mu^\delta_{\gamma(\ell)})}{\ell}
            \end{align}
            for all $x\in X$ and all length-minimising geodesics $\gamma:[0,T]\rightarrow X$ such that $\gamma(0)=x$. 
        \end{enumerate}
    \end{definition}
    Crucially the Ollivier curvature turns out to asymptotically approach the Ricci curvature in manifolds:
    \begin{fact}
        Let $\mathcal{M}$ be a manifold and $\mu$ the volume measure on $\mathcal{M}$. For any $p\in \mathcal{M}$ and any $V\in T_p\mathcal{M}$ we have
        \begin{align}
            \text{Ric}_p(V,V)=\lim_{\delta,\ell\rightarrow 0}\frac{D+2}{\delta^2}\kappa^{\delta,\ell}_p(\gamma_V)
        \end{align}
        for any $\gamma_V\in V$. $\qed$
    \end{fact}
    This suggests the following definition:
    \begin{definition}
        Let $(X,\rho_X,\mu)$ be a metric-measure space. Then the \textit{scale-free Ollivier curvature} at $x\in X$ is given
        \begin{align}
            \kappa_x(\gamma)\coloneqq \lim_{\delta,\ell\rightarrow 0}\frac{1}{\delta^2}\kappa^{\delta,\ell}_x(\gamma)
        \end{align}
        for all suitable geodesics $\gamma$ at $x$.
    \end{definition}
    Note that there is some ambiguity in the above relating to the way we take $\delta,\:\ell\rightarrow 0$; for the moment we simply take an \textit{ad hoc} approach and take the limit in such a way that the results make as much sense as possible. Indeed, for our purposes it will be sufficient to assume that $\ell>\delta$ since the Ollivier curvature will turn out to be independent of $\ell$ in this regime. The same is true up to leading order for the Ollivier curvature in manifolds (without any assumptions on $\delta$ and $\ell$) but a more complete understanding of the Ollivier curvature in various spaces is required to make the scale-free Ollivier curvature well-defined in more general scenarios.
    
    \section{The Brownian Continuum Random Tree}\label{section: BCRT}
    In this section we consider the basic properties of the Brownian continuum random tree. We begin with a review of some results relating to the metric structure of deterministic continuum trees before recalling the definition of the Gromov-Hausdorff distance and using it to introduce some suitable topologies in the space of continuum trees. We finish this section with a discussion of Brownian excursions, the characterisation of the Brownian continuum random tree and the definition of its volume form. In this section and the next we will use $\land$ and $\lor$ to denote the meet and join of subsets of $\mathbb{R}$ respectively, i.e. $\land A=\inf A$ and $\lor A=\sup A$ for any $A\subseteq \mathbb{R}$. Also as usual we extend the notation such that $s\land t=\land \set{s,t}$, and $s\lor t=\lor\set{s,t}$ for any $s,\:t\in \mathbb{R}$.
    \subsection{Metric Properties of Continuum Trees}
    \begin{definition}
        A continuum tree is a compact metric space $(\mathcal{T},\rho)$ satisfying the following properties:
        \begin{enumerate}
            \item For any $x,\:y\in \mathcal{T}$, there exists a unique geodesic between $x$ and $y$, i.e. there exists a unique isometric curve $[0,\rho(x,y)]\rightarrow \mathcal{T}$ such that $0\mapsto x$ and $\rho(x,y)\mapsto y$; the image of such a curve is denoted $[[x,y]]$. We will also use the notation $[[x,y))\coloneqq [[x,y]]\backslash \set{y}$ and $((x,y))\coloneqq [[x,y]]\backslash\set{x,y}$.
            \item The codomain of every continuous injective curve from $x$ to $y$ is $[[x,y]]$.
        \end{enumerate}
        The set of all continuum trees (up to isomorphism) is denoted $\mathbb{T}$.
    \end{definition}
    For ease of notation later on, it will be convenient to understand the concatenation of geodesic curves in terms of the image sets only. Let us begin with the following definition:
    \begin{definition}
        Let $(\mathcal{T},\rho)$ be a continuum tree; for any $x,\:y,\:z\in \mathcal{T}$ let $\gamma_1$, $\gamma_2$ and $\gamma_3$ denote respectively the unique geodesics between $x$ and $y$, $y$ and $z$ and $x$ and $z$. We write
        \begin{align}
            [[x,z]]=[[x,y]]\oplus [[y,z]]
        \end{align}
        iff $\rho(x,y)\leq \rho(x,z)$ and
        \begin{align}
            \gamma_3(t)=\left\{\begin{array}{rl}
                \gamma_1(t), &  t\in [0,\rho(x,y)]\\
                \gamma_2(t-\rho(x,y)), & t\in [\rho(x,y),\rho(x,z)]
            \end{array}\right..
        \end{align}
        We say that $x,\:y,\:z\in \mathcal{T}$ are \textit{colinear} iff one of
        \begin{align}
            [[x,y]]=[[x,z]]\oplus [[z,y]] && [[x,z]]=[[x,y]]\oplus [[y,z]] && [[y,z]]=[[y,x]]\oplus [[x,z]]
        \end{align}
        holds.
    \end{definition}
    The point is that the relation $\oplus$ and the notion of colinearity can both be understood entirely in terms of the geodesic images $[[\cdot,\cdot]]$:
    \begin{fact}
        For any $x,\:y,\:z\in \mathcal{T}$ we have $[[x,z]]=[[x,y]]\oplus [[y,z]]$ iff $[[x,y]]\cup [[y,z]]=[[x,z]]$ and $[[x,y]]\cap [[y,z]]=\set{y}$. In particular, $x,\:y$ and $z$ are colinear iff one of $x\in [[y,z]]$, $y\in [[x,z]]$ or $z\in [[x,y]]$ holds.
    \end{fact}
    \begin{proof}
        The forwards direction is trivial. For the backwards direction we note that if $\gamma_1$ and $\gamma_2$ are the geodesics associated to $[[x,y]]$ and $[[y,z]]$ respectively and $[[x,y]]\cup [[y,z]]=[[x,z]]$ we can define a continuous curve
        \begin{align}
            \gamma:[0,\rho(x,y)+\rho(y,z)]\rightarrow\mathcal{T} && \gamma(t)=\left\{\begin{array}{rl}
                \gamma_1(t), & t\in [0,\rho(x,y)] \\
                \gamma_2(t-\rho(x,y)), & [\rho(x,y),\rho(x,y)+\rho(y,z)]
            \end{array}\right.\nonumber
        \end{align}
        from $x$ to $z$; moreover, since $[[x,y]]\cap [[y,z]]=\set{y}$, $\gamma$ is injective and so its image is given by $[[x,z]]$. The statement about colinearity follows immediately.
    \end{proof}
    With this notation in place we can prove the following key structural result:
    \begin{fact}\label{fact: LambdaPointTriangle}
        Let $(\mathcal{T},\rho)$ be a continuum tree. For any three points $x,\:y,\:z\in \mathcal{T}$, there is a unique point $\Lambda(xyz)\in \mathcal{T}$ such that 
        \begin{align}
            [[x,y]]=[[x,\Lambda(xyz)]]\oplus [[\Lambda(xyz),y]] && [[x,z]]=[[x,\Lambda(xyz)]]\oplus [[\Lambda(xyz),z]] && [[y,z]]=[[y,\Lambda(xyz)]]\oplus [[\Lambda(xyz),z]].
        \end{align}
        Then (for instance) $\Lambda(xyz)=\Lambda(xyu)$ for all $u\in [[\Lambda(xyz),z]]$. If $x$, $y$ and $z$ are colinear with e.g. $z\in [[x,y]]$ we have $\Lambda(xyz)=z$. 
    \end{fact}
    \begin{proof}
        If $x$, $y$ and $z$ are colinear then the statement is trivial. Thus suppose that $x$, $y$ and $z$ are not colinear; in particular, assume that $z\notin [[x,y]]$. Concatenating the unique geodesic curves associated to $[[x,z]]$ and $[[z,y]]$ respectively we obtain a continuous curve $\gamma$ from $x$ to $y$; moreover $z$ lies in the image of $\gamma$ while $z\notin [[x,y]]$ i.e. $\gamma$ is not the unique geodesic associated to $[[x,y]]$ and so must not be injective. Thus consider any $u\in [[x,z]]\cap [[z,y]]$; since $[[x,u]]$ and $[[u,y]]$ are the continuous images of compact intervals $[0,\rho(x,u)]$ and $[0,\rho(u,y)]$ respectively, we have that $[[x,u]]\cap [[u,y]]$ is compact and so the continuous mapping $[[x,u]]\cap [[u,y]]\rightarrow [0,\infty)$ given by $v\mapsto \rho(u,v)$ realises its supremum at some point $\Lambda(xyu)\in [[x,u]]\cap [[u,y]]$; moreover this point is unique since $v\mapsto \rho(u,v)$ is the inverse of the geodesic associated to e.g. $[[u,x]]$ restricted to $[[x,u]]\cap [[u,y]]$. Also $\Lambda(xyu)=\Lambda(xyz)$ for all $u\in [[x,z]]\cap [[z,y]]$ since $u\in [[z,v]]$ for all $v\in [[x,u]]\cap [[u,y]]$, i.e. $\rho(z,v)=\rho(z,u)+\rho(u,v)$ for all $v\in [[x,u]]\cap [[u,y]]$ and $\rho(z,v)\geq \rho(z,w)$ for all $v\in [[x,u]]\cap [[u,y]]$ and all $w\in [[z,u]]\subseteq [[x,z]]\cap [[z,y]]$. Then in particular
        \begin{align}
            \sup_{v\in [[x,z]]\cap [[z,y]]}\rho(x,v)=\rho(z,u)+\sup_{v\in [[x,u]]\cap [[u,y]]}\rho(u,v)=\rho(z,u)+\rho(u,\Lambda(xyu))=\rho(z,\Lambda(xyu))\nonumber
        \end{align}
        and the equality $\Lambda(xyu)=\Lambda(xyz)$ follows from uniqueness. For the remainder of the proof it is sufficient to prove that $\Lambda(xyz)\in [[x,y]]$; but in particular, we can consider the concatenation of the paths associated to $[[x,\Lambda(xyz)]]$ and $[[\Lambda(xyz),y]]$; for any $u\in [[x,\Lambda(xyz)]]\cap [[\Lambda(xyz),y]]\subseteq [[x,z]]\cap [[z,y]]$ we have $\Lambda(xyz)\in [[u,z]]$ so $\rho(z,u)\geq \rho(z,\Lambda(xyz))$ and $u=\Lambda(xyz)$ by uniqueness of the point realising the supremum of the mapping $v\mapsto \rho(z,v)$ on $[[x,z]]\cap [[z,y]]$.
    \end{proof}
    At a slightly more explicit level we have that for any $x,\:,\:z\in X$ there is a unique point $\Lambda(xyz)$ such that:
    \begin{subequations}
        \begin{align}
            \rho(x,y)&=\rho(x,\Lambda(xyz))+\rho(\Lambda(xyz),y)\\
            \rho(x,z)&=\rho(x,\Lambda(xyz))+\rho(\Lambda(xyz),z)\\
            \rho(y,z)&=\rho(y,\Lambda(xyz))+\rho(\Lambda(xyz),z).
        \end{align}
    \end{subequations}
    Since length-minimising geodesics are unique in continuum trees, this result can also be understood as demonstrating that triangles in continuum trees are infinitely thin; an immediate heuristic consequence, which can be made formal without too much difficulty, is the following:
    \begin{fact}
        A continuum tree is $CAT(K)$ for each $K\in \mathbb{R}$. $\qed$
    \end{fact}
    Essentially the main conclusion of this work is that an analogous result holds for the Ollivier curvature when the scale-dependence is properly accounted for.
    
    It will be convenient to take the tree to be rooted and consider the associated ancestry relation:
    \begin{definition}
        A \textit{rooted continuum tree} is a triple $(\mathcal{T},\rho,\boldsymbol{rt})$ where $(\mathcal{T},\rho)$ is a continuum tree and $\boldsymbol{rt}\in \mathcal{T}$; the set of all rooted continuum trees is denoted $\mathbb{T}_{\boldsymbol{rt}}$. The \textit{ancestry relation} $\preceq$ on a rooted continuum tree $(\mathcal{T},\rho,\boldsymbol{rt})$ is defined such that
        \begin{align}
            \preceq=\Set{(x,y)\in \mathcal{T}\times \mathcal{T}:x\in [[\boldsymbol{rt},y]]}.
        \end{align}
        We say that $x$ is an \textit{ancestor} of $y$ or that $y$ is a \textit{descendent} of $x$ iff $x\preceq y$. For any $x\in \mathcal{T}$, let the \textit{descendant subtree} at $x$ be the triple $(\mathcal{T}^+(x),\rho\vert \mathcal{T}^+(x)\times \mathcal{T}^+(x),x)$ where
        \begin{align}
            \mathcal{T}^+(x)\coloneqq \set{y\in \mathcal{T}:x\preceq y}.
        \end{align}
    \end{definition}
    Clearly the ancestry relation defines a partial order on any rooted continuum tree which taken as a poset is pointed i.e. the root is the least element of the tree with respect to the ancestry relation. Also note that the descendant subtree at a point is also a rooted continuum tree. To see this it is sufficient to note that if $x\preceq y,\:z$ then $x\preceq u$ for all $u\in [[y,z]]$. This is because either $u\in [[y,\Lambda(xyz)]]$ or $u\in [[\Lambda(xyz),z]]$ and so $u\in [[x,y]]$ or $u\in [[x,z]]$. But because $x\in [[\boldsymbol{rt},y]]$ and $x\in [[\boldsymbol{rt},z]]$ we have $[[\boldsymbol{rt},y]]=[[\boldsymbol{rt},x]]\oplus [[x,y]]$ and $[[\boldsymbol{rt},z]]=[[\boldsymbol{rt},x]]\oplus [[x,z]]$ as required.

    We shall now need some results about the structure of balls and associated functionals in continuum trees. It will be useful to define the following notation:
    \begin{definition}
        Let $(\mathcal{T},\rho)$ be a continuum tree. For any distinct $x,\:y\in \mathcal{T}$ we define $\mathcal{T}^x(y)$ as the connected component of $\mathcal{T}\backslash \set{x}$ containing $y$; we then let $\mathcal{T}^{\mathbb{c}x}(y)\coloneqq \mathcal{T}\backslash \mathcal{T}^x(y)$. Then for any $\delta>0$ we also define
        \begin{align}
            \mathbb{A}_\delta^{\mathcal{T}}(x,y)\coloneqq \mathbb{B}_r^{\mathcal{T}}(x)\cap \mathcal{T}^x(y) && \mathbb{O}_\delta^{\mathcal{T}}(x,y)\coloneqq \mathbb{B}_\delta^{\mathcal{T}}(x)\cap \mathcal{T}^{\mathbb{c}x}(y)
        \end{align}
        Note that we use the symbols $\mathbb{A}$ and $\mathbb{O}$ as mnemonics for \textit{ancestry} and \textit{offspring} respectively.
    \end{definition}
    The idea of the above is to provide a decomposition of local balls into components with respect to a fixed reference point, understood to be an ancestor. In particular, if the tree is rooted and when the reference point is taken as the root, the mnemonic is entirely accurate: 
    \begin{corollary}\label{corollary: OffspringBallRoot}
        Let $(\mathcal{T},\rho,\boldsymbol{rt})$ be a rooted continuum tree. Then for any $x\in \mathcal{T}$ distinct from the root we have
        \begin{align}
            \mathbb{O}_{\delta}^{\mathcal{T}}(x,\boldsymbol{rt})=\mathbb{B}_\delta^+(x)
        \end{align}
        where $\mathbb{B}_\delta^+(x)$ is the $\delta$-ball of $x$ in the descendant subtree $\mathcal{T}^+(x)$. $\qed$
    \end{corollary}
    The decomposition itself will prove to be very useful; for instance consider the following:
    \begin{lemma}\label{lemma: BallDecomposition}
        Let $(\mathcal{T},\rho)$ be a continuum tree. For any $x,\:y\in \mathcal{T}$ such that $\rho(x,y)=\ell>0$, and any $\delta>0$ such that $\delta<\ell$ we have
        \begin{align}
            \rho(\sigma,y)\geq \ell-\rho(x,\sigma) && \rho(\tau,y)=\ell+\rho(\tau,x)
        \end{align}
        for all $\sigma\in \mathbb{A}_\delta^{\mathcal{T}}(x,y)$ and all $\tau\in \mathbb{O}_{\delta}^\mathcal{T}(x,y)$. Equality only holds in the first expression for $x$, $y$ and $\sigma$ colinear.
    \end{lemma}
    \begin{proof}
        For any $\sigma\in \mathbb{A}_{\delta}^{\mathcal{T}}(x,y)$, suppose that $x$, $y$ and $\sigma$ are colinear; then $\sigma\in [[x,y]]$ and
        \begin{align}
            \ell=\rho(x,y)=\rho(x,\sigma)+\rho(\sigma,y)\nonumber
        \end{align}
        which proves the statement. Note that we do not have $y\in [[x,\sigma]]$ since then $\rho(x,\sigma)\geq \rho(x,y)=\ell>\delta$ whilst $x\in [[\sigma,y]]$ would imply that $\sigma\notin \mathcal{T}^{x}(y)$. But if $x$, $y$ and $\sigma$ are not colinear then we have a point $\Lambda(xy\sigma)\in [[x,y]]$ such that
        \begin{align}
            \rho(x,\sigma)=\rho(x,\Lambda(xy\sigma))+\rho(\Lambda(xy\sigma),\sigma) && \rho(\sigma,y)=\rho(\sigma,\Lambda(xy\sigma))+\rho(\Lambda(xy\sigma),y) && \ell=\rho(x,y)=\rho(x,\Lambda(xy\sigma))+\rho(\Lambda(xy\sigma),y)\nonumber
        \end{align}
        by fact \ref{fact: LambdaPointTriangle}. Then
        \begin{align}
            \rho(\sigma,y)=\rho(\sigma,\Lambda(xy\sigma))+\rho(\Lambda(xy\sigma),y) =\ell-\rho(x,\Lambda(xy\sigma))+\rho(\sigma,\Lambda(xy\sigma))>\ell-\rho(x,\Lambda(xy\sigma))-\rho(\sigma,\Lambda(xy\sigma))=\ell-\rho(x,\sigma)\nonumber
        \end{align}
        as required, where the inequality is strict by definiteness since $\sigma\notin [[x,y]]$ i.e. $\sigma\neq \Lambda(xy\sigma)$. The second expression follows immediately from the fact that if $\tau\in \mathcal{T}^{\mathbb{c}x}(y)$ then $x\in [[\tau,y]]$. 
    \end{proof}
    In fact we have a more refined result:
    \begin{proposition}\label{proposition: BallIntersection}
        Let $x,\:y\in \mathcal{T}$ be distinct and let $\ell=\rho(x,y)$. Then for any $\delta,\:\varepsilon>0$ such that $\delta+\varepsilon>\ell$ let 
        \begin{align}
            r=\frac{1}{2}(\delta+\varepsilon-\ell)
        \end{align}
        and let $u,\:v$ and $w$ denote respectively the unique points of $[[x,y]]$ such that $\rho(x,u)=\ell-\varepsilon$, $\rho(x,v)=(\ell-\varepsilon+\delta)/2$ and $\rho(x,w)=\delta$. Note that $\ell-\varepsilon<(\ell-\varepsilon+\delta)/2<\delta $ and
        \begin{align}
            \mathbb{B}^{\mathcal{T}}_{\delta}(x)\cap \mathbb{B}^{\mathcal{T}}_{\varepsilon}(y)=\bigcup_{\sigma\in [[u,w]]}\mathbb{B}_{\rho(u,\sigma)\land \rho(\sigma,w)}^{\mathcal{T}}(\sigma)=\mathbb{B}_{r}^{\mathcal{T}}(v).
        \end{align}
    \end{proposition}
    \begin{proof}
        $\ell-\varepsilon<(\ell-\varepsilon+\delta)/2<\delta$ is equivalent to $\ell<(\ell+\varepsilon+\delta)/2<\delta+\varepsilon$ and the desired inequalities immediately follow from $\ell<\delta+\varepsilon$. Now for notational convenience let us denote
        \begin{align}
            \mathscr{B}\coloneqq \bigcup_{\sigma\in [[u,w]]}\mathbb{B}_{\rho(u,\sigma)\land \rho(\sigma,w)}^{\mathcal{T}}(\sigma)\nonumber.
        \end{align}
        We show
        \begin{align}
            \mathbb{B}^{\mathcal{T}}_{\delta}(x)\cap \mathbb{B}^{\mathcal{T}}_{\varepsilon}(y)\subseteq \mathscr{B}\subseteq \mathbb{B}_{r}^{\mathcal{T}}(v)\subseteq \mathbb{B}^{\mathcal{T}}_{\delta}(x)\cap \mathbb{B}^{\mathcal{T}}_{\varepsilon}(y)\nonumber.
        \end{align}
        
        $\mathbb{B}^{\mathcal{T}}_{\delta}(x)\cap \mathbb{B}^{\mathcal{T}}_{\varepsilon}(y)\subseteq \mathscr{B}$: let $z\in \mathbb{B}^{\mathcal{T}}_{\delta}(x)\cap \mathbb{B}^{\mathcal{T}}_{\varepsilon}(y)$ and consider $\sigma_z=\Lambda(xyz)$ as per fact \ref{fact: LambdaPointTriangle}. We have
        \begin{align}
            \rho(x,z)=\rho(x,\sigma_z)+\rho(\sigma_z,z) && \rho(y,z)=\rho(y,\sigma_z)+\rho(\sigma_z,z) && \ell=\rho(x,y)=\rho(x,\sigma_z)+\rho(\sigma_z,y)\nonumber
        \end{align}
        by definition; since $\rho(y,\sigma_z)\leq \rho(y,z)<\varepsilon$ we have $\rho(x,\sigma_z)>\ell -\varepsilon$ and $u\in [[x,\sigma_z]]$ while $\rho(x,\sigma_z)\leq \rho(x,z)<\delta$ and $\sigma_z\in [[x,w]]$ i.e. $\sigma_z\in [[u,w]]$. But then since $\varepsilon>\rho(y,z)$ and $\sigma_z\in [[y,u]]$, $\rho(y,u)=\rho(y,x)-\rho(x,u)=\ell-(\ell-\varepsilon)=\varepsilon$ we have
        \begin{align}
            \varepsilon>\rho(y,\sigma_z)+\rho(\sigma_z,z)=\rho(y,u)-\rho(u,\sigma_z)+\rho(\sigma_z,z)=\varepsilon-\rho(u,\sigma_z)+\rho(\sigma_z,z)\nonumber
        \end{align}
        we have $\rho(\sigma_z,z)<\rho(u,\sigma_z)$. Similarly since $\delta>\rho(x,z)$ we have
        \begin{align}
            \delta>\rho(x,\sigma_z)+\rho(\sigma_z,z)=\rho(x,w)-\rho(w,\sigma_z)+\rho(\sigma_z,z)=\delta-\rho(w,\sigma_z)+\rho(\sigma_z,z)\nonumber
        \end{align}
        i.e. $\rho(\sigma_z,z)<\rho(w,\sigma_z)$. Thus $z\in \mathbb{B}_{\rho(\sigma_z,u)\land \rho(\sigma_z,w)}^{\mathcal{T}}(\sigma_z)\subseteq \mathscr{B}$ as required.

        $\mathscr{B}\subseteq\mathbb{B}_r^{\mathcal{T}}(v)$: let $z\in \mathscr{B}$, i.e. we have a $\sigma_z\in [[u,w]]$ such that $\rho(z,\sigma_z)<\rho(u,\sigma_z)\land \rho(w,\sigma_z)$. Now noting that $v\in [[u,w]]$ with 
        \begin{align}
            \rho(u,v)=\rho(x,v)-\rho(x,u)=\frac{1}{2}(\delta+\varepsilon-\ell)=r && \rho(w,v)=\rho(x,w)-\rho(x,v)=\frac{1}{2}(\delta+\varepsilon-\ell)=r\nonumber
        \end{align}
        we see in particular that $\rho(u,v)=\rho(w,v)=r$ and $\rho(u,w)=2r$. Then
        \begin{align}
            \rho(v,z)=\rho(v,\sigma_z)+\rho(\sigma_z,z)<\rho(v,\sigma_z)+\rho(u,\sigma_z)\land \rho(w,\sigma_z)\nonumber;
        \end{align}
        if $\sigma_z\in [[u,v]]$ then $\rho(u,\sigma_z)\leq r\leq 2r-\rho(u,\sigma_z)=\rho(u,w)-\rho(u,\sigma_z)=\rho(\sigma_z,w)$ and
        \begin{align}
            \rho(v,z)<\rho(v,\sigma_z)+\rho(u,\sigma_z)=\rho(u,v)=r\nonumber.
        \end{align}
        Similarly if $\sigma_z\in [[v,w]]$ then $\rho(w,\sigma_z)\leq \rho(u,\sigma_z)$ and $\rho(v,z)<\rho(v,\sigma_z)+\rho(w,\sigma_z)=\rho(w,v)=r$. Thus $z\in \mathbb{B}_{r}^{\mathcal{T}}(v)$ and $\mathscr{B}\subseteq \mathbb{B}_{r}^{\mathcal{T}}(v)$ as required.

        $\mathbb{B}_{r}^{\mathcal{T}}(v)\subseteq \mathbb{B}_{\delta}^{\mathcal{T}}(x)\cap \mathbb{B}_{\varepsilon}^{\mathcal{T}}(y)$: let $z\in \mathbb{B}^{\mathcal{T}}_r(v)$. By subadditivity we have
        \begin{align}
            \rho(x,z)&\leq \rho(x,v)+\rho(v,z)=\frac{1}{2}(\ell-\varepsilon+\delta)+\rho(v,z)<\frac{1}{2}(\ell-\varepsilon+\delta)+r=\frac{1}{2}(\ell-\varepsilon+\delta)+\frac{1}{2}(\varepsilon+\delta-\ell)=\delta\nonumber\\
            \rho(y,z)&\leq \rho(y,v)+\rho(v,z)=\frac{1}{2}(\ell+\varepsilon-\delta)+\rho(v,z)<\frac{1}{2}(\ell+\varepsilon-\delta)+r=\frac{1}{2}(\ell+\varepsilon-\delta)+\frac{1}{2}(\varepsilon+\delta-\ell)=\varepsilon\nonumber
        \end{align}
        as required.
    \end{proof}
    We will also need the following result which essentially states that a mapping of a particular form is $1$-Lipschitz, i.e. satisfies
    \begin{align}
        \rho(f(x),f(y))\leq \rho(x,y)
    \end{align}
    for all $x$ and $y$ in the domain of $f$:
    \begin{proposition}\label{proposition: LipschitzFunction}
        Let $(\mathcal{T},\rho)$ be a continuum tree. For any distinct $x,\:y\in \mathcal{T}$ define the mapping
        \begin{align}
            f_{x,y}:\mathcal{T}\rightarrow \mathbb{R} && f_{x,y}:\sigma\mapsto f_{x,y}(\sigma)=\left\{\begin{array}{rl}
                \rho(\sigma,y), &  \sigma\in \mathcal{T}^{y}(x)\\
                -\rho_{\mathbb{e}}(\sigma,y), & \sigma\in \mathcal{T}^{\mathbb{c}y}(x)
            \end{array}\right.\nonumber.
        \end{align}
        $f_{x,y}$ is $1$-Lipschitz for all $(x,y)\in \mathcal{T}\times \mathcal{T}\backslash \triangle_{\mathcal{T}}$.
    \end{proposition}
    \begin{proof}
        By construction if $\sigma,\:\tau\in \mathcal{T}^{y}(x)$ or $\sigma,\:\tau\notin \mathcal{T}^{y}(x)$ we have
        \begin{align}
            \vert f_{x,y}(\sigma)-f_{x,y}(\tau)\vert=\vert\rho(\sigma,y)-\rho(\tau,y)\vert\nonumber.
        \end{align}
        But by the subadditivity of $\rho$ we have
        \begin{align}
            \rho(\sigma,y)&\leq \rho(\sigma,\tau)+\rho(\tau,y) & 
            \rho(\tau,y)&\leq \rho(\sigma,\tau)+\rho(\sigma,y)\nonumber
        \end{align}
        so
        \begin{align}
            \rho(\sigma,y)-\rho(\tau,y)\leq \rho(\sigma,\tau) && \rho(\tau,y)-\rho(\sigma,y)\leq \rho(\sigma,\tau)\nonumber
        \end{align}
        i.e. $\vert\rho(\sigma,y)-\rho(\tau,y)\vert\leq \rho(\sigma,\tau)$ as required. The remaining case is given by $\sigma\in \mathcal{T}^{y}(x)$ and $\tau\in \mathcal{T}^{\mathbb{c}y}(x)$. Then $y\in [[\sigma,\tau]]$ so $\rho(\sigma,\tau)=\rho(\sigma,y)+\rho(y,\tau)$ and
        \begin{align}
            \vert f_{x,y}(\sigma)-f_{x,y}(\tau)\vert=\rho(\sigma,y)+\rho(\tau,y)=\rho(\sigma,\tau)\nonumber
        \end{align}
        as required.
    \end{proof}
    \subsection{Brownian Excursions and the Brownian Continuum Random Tree}
    We now rapidly review the theory of Brownian excursions. First recall:
    \begin{definition}
        An \textit{excursion} is a continuous mapping $E:[0,\infty)\rightarrow [0,\infty)$ such that there is a $T\in (0,\infty)$ called the \textit{lifetime} of the excursion such that
        \begin{align}
            E(0)=E(s)=0 && E(t)>0
        \end{align}
        for all $s\geq T$ and all $t\in (0,T)$. The set of all excursions of lifetime $T$ is denoted $\mathscr{E}^{(T)}$ and we define the set of all excursions
        \begin{align}
            \mathscr{E}\coloneqq \bigcup_{T\in (0,\infty)}\mathscr{E}^{(T)}.
        \end{align}
        $\mathscr{E}$ will be equipped with the topology of compact convergence i.e. uniform convergence on compact subsets of the domain. We will often denote the lifetime of an excursion $E\in \mathscr{E}$ by $E(T)$ and also define
        \begin{align}
            H(E)=\sup E([0,\infty)).
        \end{align}
        Note that $H(E)$ is finite since $H=\sup E([0,T(E)])$ with $[0,T(E)]$ compact (bounded).
    \end{definition}
    Following It\^{o} \cite{Ito-PointProcesses} we may regard a standard linear Brownian motion  (more generally any recurrent Markov process) $\mathbb{w}$ as a Poisson point process in the infinite dimensional space $(0,\infty)\times \mathscr{E}$ with intensity $\lambda\times \nu$ where $\nu$ is a $\sigma$-finite Borel measure on $\mathscr{E}$ called the \textit{excursion measure}; essentially the idea is that the set
    \begin{align}
        \mathbb{w}^{-1}(\set{0})=\set{t\in [0,\infty):\mathbb{w}_t=0}
    \end{align}
    is closed and so its complement can be expressed as a union of disjoint open intervals (note that the set contains $0$) on which $\mathbb{w}$ has constant sign; that is to say $\vert\mathbb{w}\vert $ essentially defines an excursion on each such interval which we call an \textit{excursion interval}. Moreover, $\mathbb{w}^{-1}(\set{0})$ is the growth set for the local time process associated to $\mathbb{w}$ and so the local time is constant on each excursion interval. As such Brownian motion naturally defines a point process in the space $(0,\infty)\times \mathscr{E}$ where a pair $(L,E)$ is chosen according to the point process iff the excursion $E$ appears in the Brownian motion at the local time $L$, i.e. the local time on the excursion interval associated to $E$ is $L$. It\^{o}'s result is then essentially a consequence of the universality of Poisson processes amongst counting processes and the strong Markov property of Brownian motion at the the times in $\mathbb{w}^{-1}(\set{0})$. While the Poisson nature of the point process of excursions is of fundamental significance, much of the power of excursion theory additionally comes from the fact that several important properties can be drawn about $\nu$ \textit{a priori}. For more detailed summaries see e.g. Refs. \cite{Rogers-GuidedTourExcursions,PitmanYor-Excursion} and Ref. \cite[Chapter XII]{RevuzYor-ContinuousMartingalesBrownianMotion} for a fairly comprehensive technical summary of the Brownian excursion theory. For our purposes we will need the following key facts:
    \begin{enumerate}
        \item $\nu^{(1)}\coloneqq \nu(\cdot \vert T(E)=1)$ is a probability measure on $\mathscr{E}^{(1)}$.
        \item For each $\alpha\in (0,\infty)$, define the scaling transformation
        \begin{align}
            \Lambda_\alpha:\mathscr{E}\rightarrow \mathscr{E} && \Lambda_\alpha:E\mapsto \left(\Lambda_\alpha E:s\mapsto \sqrt{\alpha} E\left(\frac{s}{\alpha}\right)\right)
        \end{align}
        which clearly defines a measurable mapping $\mathscr{E}^{(1)}\rightarrow \mathscr{E}^{(\alpha)}$ for each $\alpha$. Then for each  $\alpha\in (0,\infty)$ define the probability measure
        \begin{align}
            \nu^{(\alpha)}\coloneqq (\Lambda_\alpha)_*\nu^{(1)}
        \end{align}
        on $\mathscr{E}^{(\alpha)}$. We have the following integral representation of $\nu$:
        \begin{align}
            \nu(\mathcal{E})=\frac{1}{\sqrt{2\pi}}\int^\infty_0\d\lambda(s) s^{-\frac{3}{2}}\nu^{(s)}(\mathcal{E}\cap \mathscr{E}^{(s)})
        \end{align}
        for all measurable $\mathcal{E}\subseteq\mathscr{E}$.
    \end{enumerate}
    We will be particularly interested in normalised Brownian excursions:
    \begin{definition}
        A \textit{normalised Brownian excursion} is a random element of $\mathscr{E}^{(1)}$, the space of excursions with lifetime $1$, with law $\nu^{(1)}$. We shall often denote a normalised Brownian excursion as the stochastic process $\mathbb{e}=\set{\mathbb{e}_t}_{t\in [0,\infty)}$ where $\mathbb{e}_t$ denotes the value of the excursion $\mathbb{e}$ at the time $t\in [0,\infty)$. 
    \end{definition}
    The point is that due to the scaling property of excursions listed above, the choice of normalisation is essentially conventional; in particular to prove properties about normalised Brownian excursions it is often sufficient to prove properties about general excursions chosen according to the excursion measure $\nu$ and rescaling.

    Our interest in Brownian excursions comes from the following fact:
    \begin{fact}\label{fact: ExcursionPseudometric}
        Let $E$ be an excursion of lifetime $T>0$ and define
        \begin{align}
            \rho_E:[0,T]\times [0,T]\rightarrow [0,\infty) && \rho_E(s,t)=E(s)+E(t)-2\land E
            [s\land t,s\lor t]
        \end{align}
        where we use the shorthand $E [s\land t,s\lor t]\coloneqq E([s\land t,s\lor t])$. $\rho_E$ is a pseudometric on $[0,T]$. We denote the Kolmogorov quotient of $([0,T],\rho_E)$ by $(\mathcal{T}_E,\rho_E)$; $(\mathcal{T}_E,\rho_E)$ is a continuum tree, called the \emph{continuum tree encoded by $E$}. The quotient map will be denoted by $\mathfrak{q}_E$ or simply $\mathfrak{q}$ if there is no room for confusion. 
    \end{fact}
    \begin{proof}
        See \cite[theorem 2.2]{LeGall-RandomTreesApplications}.
    \end{proof}
    Note that this encoding extends to an encoding of \textit{rooted} continuum trees $(\mathcal{T}_E,\rho_E,\boldsymbol{rt}_E)$ where $\boldsymbol{rt}_E=\mathfrak{q}(0)$. We wish to use this encoding to transfer the Brownian excursion measure $\nu$ to the set of trees $\mathbb{T}$. The latter will be equipped with the Gromov-Hausdorff topology or some natural modification when we have additional structure (root vertex/measure) \cite{BuragoBuragoIvanov-MetricGeometry,Gromov-MetricStructures,AbrahamDelmasHoscheit-NoteGromovHausdorffProkhorov}. Let us briefly recall the relevant concepts here: the Gromov-Hausdorff distance between two compact metric spaces $X$ and $Y$ is obtained by taking the infimum of the Hausdorff distance between $X$ and $Y$ regarded as compact subsets of some ambient space $Z$, that is the infimum ranges over all isometric imbeddings of $X$ and $Y$ into arbitrary metric spaces $Z$. This defines a metric on the space of isometry classes of compact metric spaces. The minimisation problem characterising the Gromov-Hausdorff distance is very difficult \textit{a priori}---the class of all triples $(Z,\iota_X,\iota_Y)$ where $Z$ is an ambient metric space and $\iota_X$ and $\iota_Y$ are isometric imbeddings of $X$ and $Y$ into $Z$ respectively is in general proper---and so it is often convenient to have an alternative characterisation. At a qualitative level, i.e. for questions regarding the topology induced by the Gromov-Hausdorff distance, it is often more convenient to work with \textit{near isometries}. In particular given metric spaces $(X,\rho_X)$ and $(Y,\rho_Y)$ and any mapping $f:X\rightarrow Y$ we define the \textit{distortion} of $f$ via
    \begin{align}
        \text{dis}(f)=\sup_{(x_1,x_2)\in X\times X}\vert \rho_X(x_1,x_2)-\rho_Y(f(x_1),f(x_2))\vert.
    \end{align}
    For any $\varepsilon>0$ a mapping $f:X\rightarrow Y$ is said to be an \textit{$\varepsilon$-isometry} iff $\text{dis}(f)\leq \varepsilon$ and for any $y\in Y$ there is an $x\in X$ such that $y\in \mathbb{B}^Y_{\varepsilon}(f(x))$. Roughly speaking $\varepsilon$-isometries do not modify distances more than an amount $\varepsilon$ and are surjective up to an error $\varepsilon$. The key point for our purposes is that if there is an $\varepsilon$-isometry $f:X\rightarrow Y$ then $\rho_{GH}(X,Y)<2\varepsilon$ where $\rho_{GH}$ denotes the Gromov-Hausdorff distance.
    \begin{fact}\label{fact: ContinuityEncodingMap}
        The encoding map $\text{enc}:E\mapsto \text{isom}(\mathcal{T}_E,\rho_E)$ is continuous when $\mathscr{E}$ is equipped with the topology of compact convergence and $\mathbb{T}$ is equipped with the Gromov-Hausdorff topology; note that for any metric space $(X,\rho_X)$, $\text{isom}(X,\rho_X)$ denotes its isometry class.
    \end{fact}
    \begin{proof}
        For any excursion $E\in \mathscr{E}$ and any $\delta>0$ consider the set
        \begin{align}
            U_\delta(E)\coloneqq \Set{\tilde{E}\in \mathscr{E}:\sup_{t\in [0,\infty)}\vert E(t)-\tilde{E}(t)\vert<\delta}\nonumber.
        \end{align}
        Obviously every net of excursions which converges to $E$ is eventually in $U_{\delta}(E)$ since excursions have compact domains; moreover for any $\tilde{E}\in U_{\delta}(E)$ we have
        \begin{align}
            \sup_{t\in [0,\infty)}\vert E(t)-\tilde{E}(t)\vert<\tilde{\delta}\nonumber
        \end{align}
        for some $\tilde{\delta}<\delta$. Thus taking $\epsilon<\delta-\tilde{\delta}$ we note that $U_\varepsilon(\tilde{E})\subseteq U_\delta(E)$ since
        \begin{align}
            \sup_{t\in [0,\infty)}\vert E(t)-\bar{E}(t)\vert \leq \sup_{t\in [0,\infty)}\vert E(t)-\tilde{E}(t)\vert +\sup_{t\in [0,\infty)}\vert \tilde{E}(t) -\bar{E}(t)\vert<\tilde{\delta}+\epsilon<\delta\nonumber
        \end{align}
        for all $\bar{E}\in U_\varepsilon(\tilde{E})$; but then every net which converges to some $\tilde{E}\in U_\delta(E)$ is eventually in $U_{\epsilon}(\tilde{E})\subseteq U_\delta(E)$ and $U_\delta(E)$ is open.
        
        Now for any $E\in \mathscr{E}$ and any $\tilde{E}\in U_\delta(E)$ we can define a mapping 
        \begin{align}
            \mathcal{T}_E\rightarrow \mathcal{T}_{\tilde{E}} && \mathfrak{q}_E(s) \mapsto \mathfrak{q}_{\tilde{E}}(\land\mathfrak{q}_E(s))\nonumber.
        \end{align}
        Using the fact that for all $\mathfrak{q}_E(s)\in \mathcal{T}_E$ we have
        \begin{align}
            \tilde{E}_{\land \mathfrak{q}_E(s) }\in (E_s-\delta,E_s+\delta)\nonumber
        \end{align}
        for all $\tilde{E}\in U_{\delta}(E)$, it is a simple consequence of subadditivity that
        \begin{align}
            \vert \rho_{E}(\mathfrak{q}_E(s) ,\mathfrak{q}_E(t))-\rho_{\tilde{E}}(\mathfrak{q}_{\tilde{E}}(\land \mathfrak{q}_E(s)),\mathfrak{q}_{\tilde{E}}(\land\mathfrak{q}_E(t)))\vert\nonumber<4\delta\nonumber. 
        \end{align}
        Also
        \begin{align}
            \rho_{\tilde{E}}(\mathfrak{q}_{\tilde{E}}(s),\mathfrak{q}_{\tilde{E}}(\land\mathfrak{q}_E(\land \mathfrak{q}_{\tilde{E}}(s))))=\tilde{E}_{s}+\tilde{E}_{\land \mathfrak{q}_E(\land \mathfrak{q}_{\tilde{E}}(s))}-2\land \tilde{E}[\land\mathfrak{q}_E(\land\mathfrak{q}_{\tilde{E}}(s)),s]<4\delta\nonumber
        \end{align}
        and the mapping $\mathfrak{q}_E(s)\mapsto \mathfrak{q}_{\tilde{E}}(\land \mathfrak{q}_E(s))$ is a $4\delta$-isometry of $\mathcal{T}_E$ into $\mathcal{T}_{\tilde{E}}$: hence we have that the Gromov-Hausdorff distance $\rho_{GH}(\mathcal{T}_E,\mathcal{T}_{\tilde{E}})<8\delta$ and so choosing $\delta<\varepsilon/8$ for each $\varepsilon>0$ shows that the encoding map is continuous since the $\varepsilon>0$ ball centred at $T_E$ contains the image of the open set $U_\delta(E)$ with respect to the encoding map.
    \end{proof}
    \begin{definition}
        The \textit{Brownian continuum random tree} is the random element $(\mathcal{T}_{\mathbb{e}},\rho_{\mathbb{e}})$ of $\mathbb{T}$ with law
        \begin{align}
            \text{law}_{(\mathcal{T}_{\mathbb{e}},\rho_{\mathbb{e}})}=\Theta && \Theta\coloneqq \text{enc}_*\nu^{(1)},
        \end{align}
        i.e. it is the random continuum tree encoded by the (random) normalised Brownian excursion $\mathbb{e}$. The \textit{rooted Brownian continuum random tree} is the random triple $(\mathcal{T}_{\mathbb{e}},\rho_{\mathbb{e}},\boldsymbol{rt}_{\mathbb{e}})$ where $\boldsymbol{rt}_{\mathbb{e}}=\mathbb{e}(0)$; strictly this is the random element of the set
        \begin{align}
            \mathbb{T}_{\boldsymbol{rt}}=\bigcup_{(\mathcal{T},\rho)\in \mathbb{T}}\set{(\mathcal{T},\rho)}\times\mathcal{T}\nonumber
        \end{align}
        with law $\Theta_{\boldsymbol{rt}}$ defined via
        \begin{align}
            (\pi_{12})_*\Theta_{\boldsymbol{rt}}=\Theta && \Theta_{\boldsymbol{rt}}(\cdot \vert \pi_{12}(\mathcal{T}_{\mathbb{e}},\rho_{\mathbb{e}},\boldsymbol{rt}_{\mathbb{e}})=(\mathcal{T}_E,\rho_E))=\lambda_E
        \end{align}
        where $\pi_{12}:(x_1,x_2,x_3)\mapsto (x_1,x_2)$ is the natural projection onto the first two elements. Henceforth we let $\mathbb{E}$ and $\mathbb{E}_{\boldsymbol{rt}}$ denote expectations with respect to the measures $\Theta$ and $\Theta_{\boldsymbol{rt}}$ respectively.
    \end{definition}
    Note that on occasion it will be convenient to regard the unnormalised analogue of the above which we shall also, somewhat carelessly call the Brownian continuum tree. Essentially the rooted Brownian continuum tree is obtained by picking a Brownian continuum tree and uniformly selecting a root vertex. 
    
    A key idea in the subsequent will be the fact that the Brownian continuum tree is invariant under rerooting:
    \begin{definition}
        For any Borel probability measure $\mu$ on $\mathbb{T}_{\boldsymbol{rt}}$ a \textit{$\mu$-invariant rerooting map} is a mapping $f:\mathbb{T}_{\boldsymbol{rt}}\rightarrow \mathbb{T}_{\boldsymbol{rt}}$ such that $f$ commutes with the projection $\pi_{12}$ and such that $f_*\mu=\mu$, with both statements holding $\mu$-almost everywhere.
    \end{definition}
    Following Croydon \cite{Croydon-VolumeGrowth} we have the following example of a $\Theta_{\boldsymbol{rt}}$-invariant rerooting map: 
    \begin{fact}\label{fact: Rerooting}
        For each $s\in (0,1)$ define the mapping
        \begin{align}
            f_s:\mathscr{E}^{(1)}\rightarrow\mathscr{E}^{(1)} && f_s:E\mapsto f_sE\nonumber
        \end{align}
        where
        \begin{align}
            f_sE(t)=\left\{\begin{array}{rl}
                \rho_E(s,s+t), &  t\in [0,1-s] \\
                \rho_E(s,s+t-1), & t\in [1-s,1]
            \end{array}\right. 
        \end{align}
        for all $t\in [0,1]$. Then we have the induced mappings
        \begin{align}
            f_{\mathfrak{q}_E(s)}:\text{enc}(\mathscr{E}^{(1)})\rightarrow \text{enc}(\mathscr{E}^{(1)}) && f_{\mathfrak{q}_E(s)}:(\mathcal{T}_E,\rho_E,\mathfrak{q}_E(0))\mapsto f_{\mathfrak{q}_E(s)}(\mathcal{T}_E,\rho_E,\mathfrak{q}_E(0))=(\mathcal{T}_{f_sE},\rho_{f_sE},\mathfrak{q}_{f_sE}(0))
        \end{align}
        and
        \begin{align}
            f_{\mathfrak{q}_E(s)}^E:\mathcal{T}_E\rightarrow \mathcal{T}_{f_sE} && f_{\mathfrak{q}_E(s)}^E:\mathfrak{q}_E(t)\mapsto \left\{\begin{array}{rl}
                \mathfrak{q}_{f_sE}(t-s+1), & t\in [0,s] \\
                \mathfrak{q}_{f_sE}(t-s), & t\in [s,1]
            \end{array}\right.;
        \end{align}
        $f_{\mathfrak{q}_E(s)}^E$ is an isometry for each $E\in \mathscr{E}^{(1)}$ and we may identify $\mathfrak{q}_{f_sE}(0)=\mathfrak{q}_{E}(s)$; then $f_s$ is a $\Theta_{\boldsymbol{rt}}$-invariant rerooting map for all $s\in (0,1)$. 
    \end{fact}
    \begin{proof}
        Fix $s\in (0,1)$ and pick $t_1,\:t_2\in [0,1]$. For each $k\in \set{1,2}$ define
        \begin{align}
            u_k\coloneqq \left\{\begin{array}{rl}
                t_k-s+1, &  t_k\in [0,s]\\
                t_k-s, & t_k\in [s,1]
            \end{array}\right.\nonumber
        \end{align}
        and note that if $t_k\in [0,s]$ then $u_k\in [1-s,1]$ while if $t_k\in [s,1]$ we have $u_k\in [0,1-s]$. The mapping $t_k\mapsto u_k$ is invertible and so we can equivalently begin by choosing $u_1,\:u_2$ and define $t_1$ and $t_2$ via this inverse. With this definition we have $f_{\mathfrak{q}_E(s)}^E(\mathfrak{q}_E(u_k))=\mathfrak{q}_{f_sE}(t_k)$. Also:
        \begin{align}
            \rho_{f_sE}(t_1,t_2)&=f_sE(t_1)+f_sE(t_2)-2\inf f_sE[t_1\land t_2,t_1\lor t_2]\nonumber\\
            &=\rho_E(s,u_1)+\rho_E(s,u_2)-2\inf \set{\rho_E(s,u):u\in [u_1\land u_2,u_1\lor u_2]}\nonumber\\
            &=\rho_E(\mathfrak{q}_E(s),\mathfrak{q}_E(u_1))+\rho_E(\mathfrak{q}_E(s),\mathfrak{q}_E(u_2))-2\inf \set{\rho_E(\mathfrak{q}_E(s),\mathfrak{q}_E(u)):u\in [u_1\land u_2,u_1\lor u_2]}\nonumber.
        \end{align}
        Noting that
        \begin{align}
            \inf \set{\rho_E(\mathfrak{q}_E(s),\mathfrak{q}_E(u)):u\in [u_1\land u_2,u_1\lor u_2]}=\rho_E(\mathfrak{q}_E(s),\Lambda_{\mathfrak{q}_E(s)\mathfrak{q}_E(u_1)\mathfrak{q}_E(u_2)})\nonumber
        \end{align}
        and recalling that
        \begin{align}
            \rho_E(\mathfrak{q}_E(s),\mathfrak{q}_E(u_k))=\rho_E(\mathfrak{q}_E(s),\Lambda_{\mathfrak{q}_E(s)\mathfrak{q}_E(u_1)\mathfrak{q}_E(u_2)})+\rho_E(\mathfrak{q}_E(u_k),\Lambda_{\mathfrak{q}_E(s)\mathfrak{q}_E(u_1)\mathfrak{q}_E(u_2)})\nonumber
        \end{align}
        for $k=1,\:2$, we immediately find that
        \begin{align}
            \rho_{f_sE}(t_1,t_2)=\rho_E(u_1,u_2) \nonumber.
        \end{align}
        Noting that $\rho_E(u_1,u_2)=0$ iff $\rho_{f_sE}(t_1,t_2)=0$, we see that the mapping $f_{\mathfrak{q}_E(s)}^E$ is well-defined; but then the fact that $\rho_{f_sE}(t_1,t_2)=\rho_E(u_1,u_2)$ for arbitrary $u_1,\:u_2\in [0,1]$ implies that
        \begin{align}
            \rho_{f_sE}(f_{\mathfrak{q}_E(s)}^E(\mathfrak{q}_E(u_1)),f_{\mathfrak{q}_E(s)}^E(\mathfrak{q}_E(u_2)))=\rho_{f_sE}(\mathfrak{q}_{f_sE}(t_1)),\mathfrak{q}_{f_sE}(t_2)))=\rho_{E}(\mathfrak{q}_{E}(u_1)),\mathfrak{q}_{E}(u_2)))\nonumber
        \end{align}
        and $f_{\mathfrak{q}_E(s)}^E$ is an isometric imbedding; surjectivity follows from our freedom to take the variables $t_k$ as our starting point.

        To see that $f_s$ is a $\Theta_{\boldsymbol{rt}}$-invariant rerooting map first note that it automatically commutes with the projection $\pi_{12}$ on $\text{end}(\mathscr{E}^{(1)})=\text{supp}(\Theta_{\boldsymbol{rt}})$ since $f_{\mathfrak{q}_E(s)}^E$ is an isometry for each $E\in \mathscr{E}^{(1)}$. To ensure that $(f_s)_*\Theta_{\boldsymbol{rt}}=\Theta_{\boldsymbol{rt}}$ it is sufficient to verify that
        \begin{align}
            (f_s)_*\Theta_{\boldsymbol{rt}}(\cdot \vert \pi_{12}(\mathcal{T}_{\mathbb{e}},\rho_{\mathbb{e}},\boldsymbol{rt}_{\mathbb{e}})=(\mathcal{T}_E,\rho_E))=\lambda_E\nonumber;
        \end{align}
        but
        \begin{align}
            (f_s)_*\Theta_{\boldsymbol{rt}}(\cdot \vert \pi_{12}(\mathcal{T}_{\mathbb{e}},\rho_{\mathbb{e}},\boldsymbol{rt}_{\mathbb{e}})=(\mathcal{T}_E,\rho_E))=(f_{\mathfrak{q}_E(s)}^E)_*\lambda_E\nonumber
        \end{align}
        and the desired result holds since $(f_{\mathfrak{q}_E(s)}^E)_*$ is an isometry.
    \end{proof}
    Note that we can trivially extend this rerooting map to unnormalised trees by first rescaling any unnormalised excursion so that it is normalised, rerooting and the applying the inverse scaling transform. As an example of the utility of rerooting invariance consider the following:
    \begin{proposition}\label{proposition: RerootingInvarianceVolume}
        Let $(\mathcal{T}_{\mathbb{e}},\rho_{\mathbb{e}})$ be the Brownian continuum random tree. Then
        \begin{align}
            \mathbb{E}(\lambda_{\mathbb{e}}(\mathbb{B}_\varepsilon^{\mathbb{e}}(x))=1-\exp\left(-2\varepsilon^2\right)
        \end{align}
        for all $\varepsilon> 0$ for $\lambda_{\mathbb{e}}$-almost all $x\in \mathcal{T}_{\mathbb{e}}$.
    \end{proposition}
    \begin{proof}
        Following \cite[Theorem 1.1]{Croydon-VolumeGrowth} we have that
        \begin{align}
            \mathbb{E}_{\boldsymbol{rt}}(\lambda_{\mathbb{e}}(\mathbb{B}_\varepsilon^{\mathbb{e}}(\boldsymbol{rt}_{\mathbb{e}}))=1-\exp\left(-2\varepsilon^2\right)\nonumber.
        \end{align}
        Consider the random variable $x=\mathfrak{q}_{\mathbb{e}}(s)$ for $s\in [0,1]$ chosen randomly according to $\lambda$ and $\mathfrak{q}_{\mathbb{e}}:[0,1]\rightarrow \mathcal{T}_{\mathbb{e}}$ the natural quotient map. Then we have a rerooting map
        \begin{align}
            f_{x}^{\mathbb{e}}(x)=\boldsymbol{rt}_{\mathbb{e}}\nonumber,
        \end{align}
        where we define $f_s$ and $f_{x}^{\mathbb{e}}=f_{\mathfrak{q}_{\mathbb{e}}}^{\mathbb{e}}$ as in fact \ref{fact: Rerooting}. Then since $f_{x}^{\mathbb{e}}$ is an isometry we have $f_{x}^{\mathbb{e}}(\mathbb{B}_\varepsilon^{\mathbb{e}}(x))=\mathbb{B}_\varepsilon^{\mathbb{e}}(\boldsymbol{rt}_{\mathbb{e}})$ so if we let $g_x:\text{enc}(\mathscr{E})\rightarrow\mathbb{R}$ be defined via
        \begin{align}
            g_x(\mathcal{T}_E,\rho_E,\boldsymbol{rt}_E)=\lambda_{E}\nonumber(\mathbb{B}_\varepsilon^{E}(x))
        \end{align}
        with $E\in \mathscr{E}$, we have
        \begin{align}
            \mathbb{E}_{\boldsymbol{rt}}\left(\lambda_{\mathbb{e}}\nonumber(\mathbb{B}_\varepsilon^{\mathbb{e}}(x))\right)&=\int_{\mathbb{T}_{\boldsymbol{rt}}}\text{d}\Theta_{\boldsymbol{rt}}(\mathcal{T}_E,\rho_E,\boldsymbol{rt}_E)g_x(\mathcal{T}_E,\rho_E,\boldsymbol{rt}_E)\nonumber\\
            &=\int_{\mathbb{T}_{\boldsymbol{rt}}}\text{d}(f_s)_*\Theta_{\boldsymbol{rt}}(\mathcal{T}_E,\rho_E,\boldsymbol{rt}_E)g_x(\mathcal{T}_E,\rho_E,\boldsymbol{rt}_E)\nonumber\\
            &=\int_{\mathbb{T}_{\boldsymbol{rt}}}\text{d}\Theta_{\boldsymbol{rt}}(\mathcal{T}_E,\rho_E,x)g_x(\mathcal{T}_E,\rho_E,f_{x}^{\mathbb{e}}(\boldsymbol{rt}_E))\nonumber\\
            &=\int_{\mathbb{T}_{\boldsymbol{rt}}}\text{d}\Theta_{\boldsymbol{rt}}(\mathcal{T}_E,\rho_E,x)\lambda_E(\mathbb{B}^E_{\varepsilon}(x))\nonumber\\
            &=\mathbb{E}_{\boldsymbol{rt}}(\lambda_{\mathbb{e}}(\mathbb{B}_\varepsilon^{\mathbb{e}}(\boldsymbol{rt}_{\mathbb{e}}))\nonumber
        \end{align}
        where in the first step we have used the definition of $g_x$, used the $\Theta_{\boldsymbol{rt}}$-rerooting invariance of $f_s$ in the second, applied the standard change of variables formula in the third and once again applied the definition of $g_x$ in the penultimate step. The final step is then an immediate consequence of the definition of $\mathbb{E}_{\boldsymbol{rt}}(\lambda_{\mathbb{e}}(\mathbb{B}_\varepsilon^{\mathbb{e}}(\boldsymbol{rt}_{\mathbb{e}}))$. But simultaneously we have:
        \begin{align}
            \mathbb{E}\left(\lambda_{\mathbb{e}}(\mathbb{B}_\varepsilon^{\mathbb{e}}(x))\right)&=\int_{\mathbb{T}}\text{d}\Theta(\mathcal{T}_E,\rho_E)\lambda_E(\mathbb{B}_\varepsilon^{E}(x))=\int_{\mathbb{T}}\text{d}\Theta(\mathcal{T}_E,\rho_E)\int_{\mathcal{T}_E}\text{d}\lambda_E(y)\lambda_E(\mathbb{B}_\varepsilon^{E}(x))=\int_{\mathbb{T}_{\boldsymbol{rt}}}\text{d}\Theta_{\boldsymbol{rt}}(\mathcal{T}_E,\rho_E,y)\lambda_E(\mathbb{B}_\varepsilon^{E}(x))\nonumber\\
            &=\mathbb{E}_{\boldsymbol{rt}}\left(\lambda_{\mathbb{e}}\nonumber(\mathbb{B}_\varepsilon^{\mathbb{e}}(x))\right)
        \end{align}
        where we have used the fact that
        \begin{align}
            \lambda_E(\mathcal{T}_E)=\int_{\mathcal{T}_E}\text{d}\lambda_E(y)=1\nonumber
        \end{align}
        for all $E\in \mathscr{E}^{(1)}$ in the second step and applied the definition of $\Theta_{\boldsymbol{rt}}$ in the third.
    \end{proof}
    We will also need the branching property of the Brownian continuum random tree; Duquesne and Le Gall provide a formulation of this branching property in terms of the local time process for general L\'{e}vy trees \cite[Theorem 4.2]{DuquesneLeGall-ProbabilisticFractalAspectsLevyTrees}. The essence of this statement is that the pairs $(x,\tilde{\mathcal{T}})\in \partial \mathbb{B}^{\mathcal{T}}_a(\boldsymbol{rt})\times \mathbb{T}_{\boldsymbol{rt}}$---where $x$ is identified with the root of $\tilde{\mathcal{T}}$---issuing away from a level set $\partial \mathbb{B}^{\mathcal{T}}_a(\boldsymbol{rt})$, $a\geq 0$, are distributed according to a Poisson point measure with density $\ell^a\times \Xi$, where $\ell^a$ is the local time at $a$ and $\Xi$ is the distribution of the L\'{e}vy tree in question; in particular, it should be stressed that the Poisson point measure is unique in distribution and independent of the structure of $\mathbb{B}^{\mathcal{T}}_a(\boldsymbol{rt})\cup\partial \mathbb{B}^{\mathcal{T}}_a(\boldsymbol{rt})$. For \textit{Brownian} trees we can make a slightly stronger statement 
    \begin{fact}\label{fact: Branching}
        Let $(\mathcal{T},\rho,\boldsymbol{rt})$ be a random tree with law $\text{enc}_*\nu$ and let $\lambda$ be the natural measure on $\mathcal{T}$ i.e. the pushforwards of the normalised Lebesgue measure on any excursion encoding $\mathcal{T}$ with respect to the quotient map $\mathfrak{q}$. The descendant subtrees at $x$ and $y$ are identically distributed for $\lambda$-almost all $x,\:y\in \mathcal{T}$.
    \end{fact}
    \begin{proof}
        The essence of theorem 4.2 of Ref. \cite{DuquesneLeGall-RandomTreesLevyProcesses} is that we have a family of \textit{local time measures} $\ell^a$, $a\geq 0$, which are supported on the sets $\partial \mathbb{B}^{\mathcal{T}}_a(\boldsymbol{rt})$  such that the rooted trees issuing from $\partial \mathbb{B}^{\mathcal{T}}_a(\boldsymbol{rt})$  are distributed according to the (unique in distribution) Poisson point measure on $\partial \mathbb{B}^{\mathcal{T}}_a(\boldsymbol{rt})\times \mathbb{T}_{\boldsymbol{rt}}$ with density $\ell^a\times \text{enc}_*\nu$. Strictly speaking, for any $a\geq 0$ we have a set $\mathcal{T}\backslash A$, $A\coloneqq \mathbb{B}^{\mathcal{T}}_a(\boldsymbol{rt})\cup \partial \mathbb{B}^{\mathcal{T}}_a(\boldsymbol{rt})$, which generically splits into multiple connected components. Each connected component consists of a random real tree such that all elements have the same ancestor $x\in \partial \mathbb{B}^{\mathcal{T}}_a(\boldsymbol{rt})$ so attaching this ancestor to the connected component as the root of the tree gives a random rooted real tree $\mathcal{T}_x$. Obviously $\mathcal{T}_x$ is a subset of $\mathcal{T}^{+}(x)$, the descendant subtree  at $x$, but may not coincide the latter if $\mathcal{T}^{+}(x)\backslash\set{x}$ is not connected. This latter event occurs, however, only if $x$ corresponds to the local minimum of some excursion such that $\mathfrak{q}(0)=\boldsymbol{rt}$; since a Brownian excursion has unique local minima almost surely while conditionally on $H(\mathcal{T})=\sup_{u\in \mathcal{T}}\rho(u,\boldsymbol{rt})\geq a$ every Brownian excursion hits the level $a$ a countable number of times (Brownian motion is an instantaneous Markov process) and we find that the trees $\mathcal{T}_x$ correspond to descendant subtrees at $x$ almost surely. That is to say the pairs $(x,\mathcal{T}^{+}(x))$ are distributed according to the Poisson point measure with density $\ell^a\times \text{enc}_*\nu$. Since this measure is unique in distribution, the distribution for the pairs $(x,\mathcal{T}^{+}(x))$ is independent of the set $A$, which is the essence of the branching property in general L\'{e}vy trees. Thus conditionally on the existence of a nontrivial descendant subtree at $x$, the distribution for descendant subtrees at points $x\in \mathcal{T}$ is simply $\text{enc}_*\nu$, independently of $x$; since a descendant subtree at $x\in \mathcal{T}_{\mathbb{e}}$ exists with probability $1$ we have the desired result.
    \end{proof}
    \section{Computing Curvature Bounds}
    In this section we compute bounds on the Ollivier curvature in the Brownian continuum random tree. We shall need two lemmas, the first of which tells us how expectations of random functions of the distance are related when evaluated on various parts of small balls:
    \begin{lemma}\label{lemma: FundamentalLemma}
        Let $(\mathcal{T}_{\mathbb{e}},\rho_{\mathbb{e}})$ be a rooted Brownian continuum random tree and $f:\mathbb{R}\rightarrow \mathbb{R}$ a measurable mapping.
        \begin{enumerate}
            \item For $\lambda_{\mathbb{e}}\times \lambda_{\mathbb{e}}$-almost all $(x,y)\in \mathcal{T}_{\mathbb{e}}\times \mathcal{T}_{\mathbb{e}}$ we have
            \begin{align}
                \mathbb{E}\left(\lambda_{\mathbb{e}}\left(\mathbb{1}_{\mathbb{B}^{\mathbb{e}}_\varepsilon(x)}f(\sigma\mapsto \rho_{\mathbb{e}}(x,\sigma))\right)\right)=\mathbb{E}\left(\lambda_{\mathbb{e}}\left(\mathbb{1}_{\mathbb{B}^{\mathbb{e}}_\varepsilon(y)}f(\sigma\mapsto \rho_{\mathbb{e}}(y,\sigma))\right)\right)
            \end{align}
            for all $\varepsilon>0$.
            \item For $\lambda_{\mathbb{e}}\times \lambda_{\mathbb{e}}$-almost all $(x,y)$ and $(\tilde{x},\tilde{y})\in \mathcal{T}_{\mathbb{e}}\times \mathcal{T}_{\mathbb{e}}$ we have
            \begin{align}
                \mathbb{E}\left(\lambda_{\mathbb{e}}\left(\mathbb{1}_{\mathbb{O}^{\mathbb{e}}_\varepsilon(x,y)}f(\sigma\mapsto \rho_{\mathbb{e}}(x,\sigma))\right)\right)=\mathbb{E}\left(\lambda_{\mathbb{e}}\left(\mathbb{1}_{\mathbb{O}^{\mathbb{e}}_\varepsilon(\tilde{x},\tilde{y})}f(\sigma\mapsto \rho_{\mathbb{e}}(\tilde{x},\sigma))\right)\right)
            \end{align}
            for all $\varepsilon>0$.
            \item For $\lambda_{\mathbb{e}}\times \lambda_{\mathbb{e}}$-almost all $(x,y)$ and $(\tilde{x},\tilde{y})\in \mathcal{T}_{\mathbb{e}}\times \mathcal{T}_{\mathbb{e}}$ we have
            \begin{align}
                \mathbb{E}\left(\lambda_{\mathbb{e}}\left(\mathbb{1}_{\mathbb{A}^{\mathbb{e}}_\varepsilon(x,y)}f(\sigma\mapsto \rho_{\mathbb{e}}(x,\sigma))\right)\right)=\mathbb{E}\left(\lambda_{\mathbb{e}}\left(\mathbb{1}_{\mathbb{A}^{\mathbb{e}}_\varepsilon(\tilde{x},\tilde{y})}f(\sigma\mapsto \rho_{\mathbb{e}}(\tilde{x},\sigma))\right)\right)
            \end{align}
            for all $\varepsilon>0$.
            \item For $\lambda_{\mathbb{e}}\times \lambda_{\mathbb{e}}$-almost all $(x,y)\in \mathcal{T}_{\mathbb{e}}\times \mathcal{T}_{\mathbb{e}}$ we have
            \begin{align}
                \mathbb{E}\left(\lambda_{\mathbb{e}}\left(\mathbb{1}_{\mathbb{O}^{\mathbb{e}}_\varepsilon(x,y)}f(\sigma\mapsto \rho_{\mathbb{e}}(x,\sigma))\right)\right)=\mathbb{E}\left(\lambda_{\mathbb{e}}\left(\mathbb{1}_{\mathbb{A}^{\mathbb{e}}_\varepsilon(x,y)}f(\sigma\mapsto \rho_{\mathbb{e}}(x,\sigma))\right)\right)
            \end{align}
            for all sufficiently small $\varepsilon>0$.
        \end{enumerate}
    \end{lemma}
    \begin{proof}
        \leavevmode
        \begin{enumerate}
            \item The proof of this statement uses rerooting invariance and is very similar to the proof of proposition \ref{proposition: RerootingInvarianceVolume}. Indeed noting that it is sufficient to prove the statement for $\lambda_{\mathbb{e}}$-almost all $x\in \mathcal{T}_{\mathbb{e}}$ in the case that $y=\boldsymbol{rt}_{\mathbb{e}}$ and letting $g_x:\text{enc}(\mathscr{E})\rightarrow \mathbb{R}$ be defined as
            \begin{align}
                g_x(\mathcal{T}_E,\rho_E,\boldsymbol{rt}_E)=\lambda_E(\mathbb{1}_{\mathbb{B}^{E}_\varepsilon(x)}f(\sigma\mapsto \rho_{E}(x,\sigma)))\nonumber,
            \end{align}
            the proof proceeds, \textit{mutatis mutandis}, precisely as in proposition \ref{proposition: RerootingInvarianceVolume}.
            \item It is sufficient to prove that 
            \begin{align}
                \mathbb{E}\left(\lambda_{\mathbb{e}}\left(\mathbb{1}_{\mathbb{O}^{\mathbb{e}}_\varepsilon(x,y)}f(\sigma\mapsto \rho_{\mathbb{e}}(x,\sigma))\right)\right)=\mathbb{E}\left(\lambda_{\mathbb{e}}\left(\mathbb{1}_{\mathbb{B}^{+}_\varepsilon(x)}f(\sigma\mapsto \rho_{\mathbb{e}}(x,\sigma))\right)\right)\nonumber
            \end{align}
            for $\lambda_{\mathbb{e}}$-almost all $x\in \mathcal{T}_{\mathbb{e}}$, where $\mathbb{B}^{+}_\varepsilon(x)\coloneqq \mathbb{B}^{\mathbb{e}}_\varepsilon(x)\cap \mathcal{T}^+(x)$ is the $\varepsilon$-ball of the root in the descendant subtree at $x$. To see this sufficiency we note that the right-hand side is uniquely determined for almost all $x$: the quantity inside the expectation is totally determined by the rooted descendant subtree at $x$ and so the expectation is uniquely determined if all descendant subtrees have the same law. But this is the case by the branching property, with the caveat that this holds conditionally on $\mathbb{B}^{\mathbb{e}}_a(\boldsymbol{rt}_{\mathbb{e}})\cup \partial \mathbb{B}^{\mathbb{e}}_a(\boldsymbol{rt}_{\mathbb{e}})$ and $\sup_{u\in \mathcal{T}_{\mathbb{e}}}\rho_{\mathbb{e}}(\boldsymbol{rt}_{\mathbb{e}},u)>a$ where $a=\rho_{\mathbb{e}}(\boldsymbol{rt}_{\mathbb{e}},x)$. $\sup_{u\in \mathcal{T}_{\mathbb{e}}}\rho_{\mathbb{e}}(\boldsymbol{rt}_{\mathbb{e}},x)>a$ is a probability one event as otherwise $x$ would have to correspond to the almost surely unique maximum of the corresponding excursion; furthermore we note that the conditional distribution on $\mathbb{B}^{\mathbb{e}}_a(\boldsymbol{rt}_{\mathbb{e}})\cup \partial \mathbb{B}^{\mathbb{e}}_a(\boldsymbol{rt}_{\mathbb{e}})$ does not depend on $\mathbb{B}^{\mathbb{e}}_a(\boldsymbol{rt}_{\mathbb{e}})\cup \partial \mathbb{B}^{\mathbb{e}}_a(\boldsymbol{rt}_{\mathbb{e}})$ (see fact \ref{fact: Branching}) and we see that the right-hand side is indeed uniquely determined for almost all $x$. Strictly speaking this argument holds for general Brownian excursions but rescaling ensures that the right-hand side is uniquely determined for normalised excursions. To see that the two expectations are in fact the same for normalised excursions let us reroot to $y$ as in part (i)---$y$ is the image under the natural quotient of some $s\in [0,1]$ chosen uniformly at random---and recall corollary \ref{corollary: OffspringBallRoot}. 
            \item This follows immediately from parts (i) and (ii) once we note that
            \begin{align}
                \mathbb{1}_{\mathbb{A}^{\mathbb{e}}_\varepsilon(x,y)}&=\mathbb{1}_{\mathbb{B}^{\mathbb{e}}_\varepsilon(x)}-\mathbb{1}_{\mathbb{0}^{\mathbb{e}}_\varepsilon(x,y)}\nonumber
            \end{align}
            for all $y\in \mathcal{T}_{\mathbb{e}}$ for all $x\in \mathcal{T}_{\mathbb{e}}$.
            \item Let $\mathbb{e}$ be a normalised Brownian excursion and choose $s$ and $t$ uniformly at randomly in $[0,1]$; let $x$ and $y$ be the points in the Brownian continuum random tree $\mathcal{T}_{\mathbb{e}}$ correspond to $s$ and $y$ respectively. We may reroot to $y$ without loss of generality so that the claim becomes 
            \begin{align}
                \mathbb{E}\left(\lambda_{\mathbb{e}}\left(\mathbb{1}_{\mathbb{B}^{+}_\varepsilon(x)}f(\sigma\mapsto \rho_{\mathbb{e}}(x,\sigma))\right)\right)=\mathbb{E}\left(\lambda_{\mathbb{e}}\left(\mathbb{1}_{\mathbb{A}^{\mathbb{e}}_\varepsilon(x,\boldsymbol{rt})}f(\sigma\mapsto \rho_{\mathbb{e}}(x,\sigma))\right)\right).\nonumber
            \end{align}
            Clearly by rescaling we see that the same holds for arbitrary unnormalised Brownian excursions; let us remove the normalisation assumption for the moment. Let $a=\rho_{\mathbb{e}}(x,\boldsymbol{rt})$. With unit probability, the point $x\in \mathcal{T}_{\mathbb{e}}$ corresponds to a pair of points $s_1,\:s_2\in [0,1]$ such that $s_1<s_2$, $\mathbb{e}(s_1)=\mathbb{e}(s_2)=a$ and the interval $(s_1,s_2)$ is an interval for an excursion $\mathbb{e}_1$ above $a$. The intervals $[0,s_1]$ and $[s_2,1]$ then define two independent Brownian motions starting from $0$ and $a$ respectively killed upon hitting $a$ and $0$ respectively. Gluing them together clearly gives a Brownian excursion with height at least $a$; let $\mathbb{e}_2$ denote the excursion obtained after rerooting to the gluing point. By the strong Markov property, $\mathbb{e}_1$ and $\mathbb{e}_2$ are independent unnormalised Brownian excursions such that $\mathbb{B}^{+}_{\varepsilon}(x)=\mathbb{B}^{1}_{\varepsilon}(\boldsymbol{rt}_1)$ and $\mathbb{A}^{\mathbb{e}}_{\varepsilon}(x,\boldsymbol{rt})=\mathbb{B}^{2}_{\varepsilon}(\boldsymbol{rt}_2)$ where the sub and superscripts $1$ and $2$ denote that the balls are in the rooted Brownian continuum random trees $(\mathcal{T}_{\mathbb{e}_1},\boldsymbol{rt}_1)$ and $(\mathcal{T}_{\mathbb{e}_2},\boldsymbol{rt}_2)$ associated to the Brownian excursions $\mathbb{e}_1$ and $\mathbb{e}_2$ respectively. But since the laws of $\mathbb{e}_1$ and $\mathbb{e}_2$ are identical, so are the laws of $\mathbb{B}^{1}_{\varepsilon}(\boldsymbol{rt}_1)$ and $\mathbb{B}^{2}_{\varepsilon}(\boldsymbol{rt}_2)$ and the statement holds for unnormalised Brownian excursions. Rescaling ensures the statement holds for normalised excursions.
        \end{enumerate}
    \end{proof}
    \begin{corollary}\label{corollary: Volumes}
        Let $(\mathcal{T}_{\mathbb{e}},\rho_{\mathbb{e}})$ be the (normalised) Brownian continuum random tree. Then
        \begin{align}
            \mathbb{E}(\lambda_{\mathbb{e}}(\mathbb{O}^{\mathbb{e}}_{\varepsilon}(x,y)))=\mathbb{E}(\lambda_{\mathbb{e}}(\mathbb{A}^{\mathbb{e}}_{\varepsilon}(x,y)))=\frac{1}{2}(1-\exp(-2\varepsilon^2)).\nonumber
        \end{align}
    \end{corollary}
    \begin{proof}
        The first equality follows from part (iv) of lemma \ref{lemma: FundamentalLemma} and so the second follows from proposition \ref{proposition: RerootingInvarianceVolume}  as long as we note that $\lambda_{\mathbb{e}}(\mathbb{B}^{\mathbb{e}}_{\varepsilon}(x))=\lambda_{\mathbb{e}}(\mathbb{O}^{\mathbb{e}}_{\varepsilon}(x,y))+\lambda_{\mathbb{e}}(\mathbb{A}^{\mathbb{e}}_{\varepsilon}(x,y))$.
    \end{proof}
    The next lemma is essentially technical and constitutes the main difficulty in the Ollivier curvature bound computation:
    \begin{lemma}\label{lemma: RecursiveBound}
        Let $(\mathcal{T}_{\mathbb{e}},\rho_{\mathbb{e}})$ be a normalised Brownian continuum random tree. For any $x,\:y\in \mathcal{T}_{\mathbb{e}}$ let $\ell=\rho_{\mathbb{e}}(x,y)$ and for any $\delta\in (0,\ell)$ let $z_{\delta}$ denote the unique element of $[[x,y]]$ such that $\rho_{\mathbb{e}}(x,z_{\delta})=\delta/2$. Then 
        \begin{align}
            \int_{\mathbb{B}_\delta^{\mathbb{e}}(x)}\d\lambda_{\mathbb{e}}(\sigma)\rho_{\mathbb{e}}(\sigma,y)\geq \lambda_{\mathbb{e}}(\mathbb{B}^{\mathbb{e}}_{\delta}(x)\backslash \mathbb{B}^{\mathbb{e}}_{\delta/2}(z_{\delta}))\ell+\frac{1}{2}\delta\lambda_{\mathbb{e}}(\mathbb{O}_\delta^{\mathbb{e}}(x,y)\backslash\mathbb{O}_{\delta/2}^{\mathbb{e}}(x,y))+\int_{\mathbb{B}^{\mathbb{e}}_{\delta/2}(z_{\delta})}\d\lambda_{\mathbb{e}}(\sigma)\rho_{\mathbb{e}}(\sigma,y).
        \end{align}
    \end{lemma}
    \begin{proof}
        By proposition \ref{proposition: BallIntersection}, the set of all points of $\mathbb{B}_\delta^{\mathbb{e}}(x)$ that lie within a distance $\ell$ of $y$ is the set $\mathbb{B}_{\delta/2}^{\mathbb{e}}(z_\delta)$; also note that $\mathbb{B}_{\delta/2}^{\mathbb{e}}(z_{\delta})\subseteq \mathbb{A}_\delta^{\mathbb{e}}(x,y)$ trivially. This implies
        \begin{align}
            \mathbb{B}_\delta^{\mathbb{e}}(x)=\mathbb{O}_\delta^{\mathbb{e}}(x,y)\cup (\mathbb{A}_\delta^{\mathbb{e}}(x,y)\backslash\mathbb{B}_{\delta/2}^{\mathbb{e}}(z_{\delta}))\cup \mathbb{B}_{\delta/2}^{\mathbb{e}}(z_{\delta})\nonumber
        \end{align}
        and we find that
        \begin{align}
            \int_{\mathbb{B}_\delta^{\mathbb{e}}(x)}\d\lambda_{\mathbb{e}}(\sigma)\rho_{\mathbb{e}}(\sigma,y)&=\int_{\mathbb{O}_\delta^{\mathbb{e}}(x,y)}\d\lambda_{\mathbb{e}}(\sigma)\rho_{\mathbb{e}}(\sigma,y)+\int_{\mathbb{A}_\delta^{\mathbb{e}}(x,y)\backslash \mathbb{B}^{\mathbb{e}}_{\delta/2}(z_{\delta})}\d\lambda_{\mathbb{e}}(\sigma)\rho_{\mathbb{e}}(\sigma,y)+\int_{\mathbb{B}^{\mathbb{e}}_{\delta/2}(z_{\delta})}\d\lambda_{\mathbb{e}}(\sigma)\rho_{\mathbb{e}}(\sigma,y)\nonumber.
        \end{align}
        Let us refer to the three terms on the right-hand side as $A$, $B$ and $C$ respectively; by lemma \ref{lemma: BallDecomposition} we have that
        \begin{align}
            A=\int_{\mathbb{O}_\delta^{\mathbb{e}}(x,y)}\d\lambda_{\mathbb{e}}(\sigma)(\ell+\rho_{\mathbb{e}}(\sigma,x))=\ell \lambda_{\mathbb{e}}(\mathbb{O}_\delta^{\mathbb{e}}(x,y))+\int_{\mathbb{O}_\delta^{\mathbb{e}}(x,y)}\d\lambda_{\mathbb{e}}(\sigma)\rho_{\mathbb{e}}(\sigma,x)\nonumber.
        \end{align}
        Similarly by the above discussion of $\mathbb{B}_{\delta/2}^{\mathbb{e}}(z_\delta)$ (c.f. proposition \ref{proposition: BallIntersection}) we have that
        \begin{align}
            B\geq \ell \lambda_{\mathbb{e}}(\mathbb{A}_\delta^{\mathbb{e}}(x,y)\backslash \mathbb{B}^{\mathbb{e}}_{\delta/2}(z_{\delta}))\nonumber.
        \end{align}
        Thus noting that the left-hand side is equal to $A+B+C$ we find that
        \begin{align}
            \int_{\mathbb{B}_\delta^{\mathbb{e}}(x)}\d\lambda_{\mathbb{e}}(\sigma)\rho_{\mathbb{e}}(\sigma,y)\geq\lambda_{\mathbb{e}}(\mathbb{B}^{\mathbb{e}}_{\delta}(x)\backslash \mathbb{B}^{\mathbb{e}}_{\delta/2}(z_{\delta}))\ell+\int_{\mathbb{O}_\delta^{\mathbb{e}}(x,y)}\d\lambda_{\mathbb{e}}(\sigma)\rho_{\mathbb{e}}(\sigma,x)+\int_{\mathbb{B}^{\mathbb{e}}_{\delta/2}(z_{\delta})}\d\lambda_{\mathbb{e}}(\sigma)\rho_{\mathbb{e}}(\sigma,y)\nonumber.
        \end{align}
        The only term that differs from the desired result is the second term; but since $\rho_{\mathbb{e}}(x,\sigma)\geq \delta/2$ for all $\sigma\in \mathbb{O}_\delta^{\mathbb{e}}(x,y)\backslash\mathbb{O}_{\delta/2}^{\mathbb{e}}(x,y)$ and $\lambda_{\mathbb{e}}(\mathbb{1}_{\mathbb{O}_{\delta/2}^{\mathbb{e}}(x,y)}(\sigma\mapsto \rho_{\mathbb{e}}(x,\sigma)))\geq 0$ we see immediately that
        \begin{align}
            \int_{\mathbb{O}_\delta^{\mathbb{e}}(x,y)}\d\lambda_{\mathbb{e}}(\sigma)\rho_{\mathbb{e}}(\sigma,x)\geq \frac{1}{2}\int_{\mathbb{O}_\delta^{\mathbb{e}}(x,y)\backslash \mathbb{O}_{\delta/2}^{\mathbb{e}}(x,y)}\d\lambda_{\mathbb{e}}(\sigma)\delta =\frac{1}{2}\delta\lambda_{\mathbb{e}}(\mathbb{O}_\delta^{\mathbb{e}}(x,y)\backslash\mathbb{O}_{\delta/2}^{\mathbb{e}}(x,y))\nonumber
        \end{align}
        as required.
    \end{proof}
    Finally we turn to the actual computation of the Ollivier curvature bound, via computations for the analogous bounds for the Wasserstein distance. Note that we shall need one fact which is entirely elementary but somewhat tedious to compute; we relegate this fact to an appendix.
    \begin{theorem}\label{theorem: MainTheorem}
        Let $(\mathcal{T}_{\mathbb{e}},\rho_{\mathbb{e}})$ be the normalised Brownian continuum random tree. For $\lambda_{\mathbb{e}}$-almost all $x,\:y\in \mathcal{T}_{\mathbb{e}}$ we have for all $\varepsilon>0$ that
        \begin{align}
            \rho_{\mathbb{e}}(x,y)+\left(\frac{19}{128}-\varepsilon\right)\delta<\mathcal{W}_{\mathbb{e}}(\mu_x^\delta,\mu_y^\delta)\leq  \rho_{\mathbb{e}}(x,y)+2\delta
        \end{align}
        for all sufficiently small $\delta\in (0,\infty)$.
    \end{theorem}
    \begin{proof}
        For convenience, throughout this proof we shall let $\ell$ denote $\rho_{\mathbb{e}}(x,y)$. To obtain the upper bound simply note that if $u\in \mathbb{B}^{\mathbb{e}}_{\delta}(x)$ and $v\in \mathbb{B}^{\mathbb{e}}_{\delta}(y)$ then we have
        \begin{align}
            \rho_{\mathbb{e}}(u,v)\leq \rho_{\mathbb{e}}(u,x)+\rho_{\mathbb{e}}(x,y)+\rho_{\mathbb{e}}(y,v)\leq \ell+2\delta\nonumber
        \end{align}
        by subadditivity trivially. Hence for any transport plan $\xi\in \Pi(\mu^\delta_x,\mu^\delta_y)$ we have
        \begin{align}
            \mathcal{W}_{\mathbb{e}}(\mu_x^\delta,\mu_y^\delta)\leq \mathcal{W}_{\mathbb{e}}(\xi)=\int_{\mathcal{T}_{\mathbb{e}}\times \mathcal{T}_{\mathbb{e}}}\text{d}\xi(u,v)\rho_{\mathbb{e}}(u,v)\leq (\ell+2\delta)\xi(\mathcal{T}_{\mathbb{e}}\times \mathcal{T}_{\mathbb{e}})=\ell+2\delta\nonumber
        \end{align}
        where the final step follows since $\xi$ is a probability measure. 
        
        We now turn to the lower bound. Recall that by the Kantorovitch duality theorem we have
        \begin{align}
            \mathcal{W}_{\mathbb{e}}(\mu_x^\delta,\mu_y^\delta)\geq \vert\mathcal{K}_{\mathbb{e}}^{\delta,x,y}(f)\vert\nonumber
        \end{align}
        for any $1$-Lipschitz $f:\mathcal{T}_{\mathbb{e}}\rightarrow \mathbb{R}$, where 
        \begin{align}
            \mathcal{K}_{\mathbb{e}}^{\delta,x,y}(g)=\int_{\mathcal{T}_{\mathbb{e}}}\d \mu^\delta_xg-\int_{\mathcal{T}_{\mathbb{e}}}\d \mu^\delta_yg=\frac{1}{\lambda_{\mathbb{e}}(\mathbb{B}_\delta^{\mathbb{e}}(x))}\int_{\mathbb{B}_\delta^{\mathbb{e}}(x)}\d\lambda_{\mathbb{e}}(\sigma)g(\sigma)-\frac{1}{\lambda_{\mathbb{e}}(\mathbb{B}_\delta^{\mathbb{e}}(y))}\int_{\mathbb{B}_\delta^{\mathbb{e}}(y)}\d\lambda_{\mathbb{e}}(\tau)g(\tau)\nonumber
        \end{align}
        for any measurable $g:\mathcal{T}_{\mathbb{e}}\rightarrow \mathbb{R}$. Also note that for any $f:\mathcal{T}_{\mathbb{e}}\rightarrow\mathbb{R}$ we have
        \begin{align}
            \mathcal{K}_{\mathbb{e}}^{\delta,x,y}(-f)=-\mathcal{K}_{\mathbb{e}}^{\delta,x,y}(f)\nonumber
        \end{align}
        so
        \begin{align}
            \vert\mathcal{K}_{\mathbb{e}}^{\delta,x,y}(f)\vert=\mathcal{K}_{\mathbb{e}}^{\delta,x,y}(f)\lor \mathcal{K}_{\mathbb{e}}^{\delta,x,y}(-f)\nonumber.
        \end{align}
        
        Let us now consider $\mathcal{K}_{\mathbb{e}}^{\delta,x,y}(f_{x,y}^{\mathbb{e}})$ with $f_{x,y}^{\mathbb{e}}$ as in proposition \ref{proposition: LipschitzFunction}:
        \begin{align}
            \mathcal{K}_{\mathbb{e}}^{\delta,x,y}(f_{x,y}^{\mathbb{e}})&=\frac{1}{\lambda_{\mathbb{e}}(\mathbb{B}_\delta^{\mathbb{e}}(x))}\int_{\mathbb{B}_\delta^{\mathbb{e}}(x)}\d\lambda_{\mathbb{e}}(\sigma)f_{x,y}^{\mathbb{e}}(\sigma)-\frac{1}{\lambda_{\mathbb{e}}(\mathbb{B}_\delta^{\mathbb{e}}(y))}\int_{\mathbb{B}_\delta^{\mathbb{e}}(y)}\d\lambda_{\mathbb{e}}(\tau)f_{x,y}^{\mathbb{e}}(\tau)\nonumber\\
            &=\frac{1}{\lambda_{\mathbb{e}}(\mathbb{B}_\delta^{\mathbb{e}}(x))}\int_{\mathbb{B}_\delta^{\mathbb{e}}(x)}\d\lambda_{\mathbb{e}}(\sigma)\rho_{\mathbb{e}}(\sigma,y)-\frac{1}{\lambda_{\mathbb{e}}(\mathbb{B}_\delta^{\mathbb{e}}(y))}\left(\int_{\mathbb{0}_\delta^{\mathbb{e}}(y,x)}\d\lambda_{\mathbb{e}}(\tau)\rho_{\mathbb{e}}(\tau,y)-\int_{\mathbb{A}_\delta^{\mathbb{e}}(y,x)}\d\lambda_{\mathbb{e}}(\tau)\rho_{\mathbb{e}}(\tau,y)\right)\nonumber.
        \end{align}
        By point (iv) of lemma \ref{lemma: FundamentalLemma}, the second term will vanish under the expectation so it is sufficient to consider the first term:
        \begin{align}
            \frac{1}{\lambda_{\mathbb{e}}(\mathbb{B}_\delta^{\mathbb{e}}(x))}\int_{\mathbb{B}_\delta^{\mathbb{e}}(x)}\d\lambda_{\mathbb{e}}(\sigma)\rho_{\mathbb{e}}(\sigma,y)\nonumber.
        \end{align}
        Let $z_{0}$ and $z_1$ respectively denote the unique points of $[[x,y]]$ such that $\rho_{\mathbb{e}}(x,z_{0})=\delta/2$ and $\rho_{\mathbb{e}}(x,z_{1})=3\delta/4$. Noting that the final term on the right-hand side of lemma \ref{lemma: RecursiveBound} has the same form as the left-hand side with a shifted centre and half the radius, we may apply the lemma twice where $z_0$ plays the role of $z_{\delta}$ the first time and $z_1$ plays the role of $z_{\delta}$ the second. We thus obtain
        \begin{align}\label{inequality: MainTheorem1}
            \int_{\mathbb{B}_\delta^{\mathbb{e}}(x)}\d\lambda_{\mathbb{e}}(\sigma)\rho_{\mathbb{e}}(\sigma,y)&\geq \lambda_{\mathbb{e}}(\mathbb{B}^{\mathbb{e}}_{\delta}(x)\backslash \mathbb{B}^{\mathbb{e}}_{\delta/2}(z_{0}))\ell+\frac{1}{2}\delta\lambda_{\mathbb{e}}(\mathbb{O}_\delta^{\mathbb{e}}(x,y)\backslash\mathbb{O}_{\delta/2}^{\mathbb{e}}(x,y))+\int_{\mathbb{B}^{\mathbb{e}}_{\delta/2}(z_{0})}\d\lambda_{\mathbb{e}}(\sigma)\rho_{\mathbb{e}}(\sigma,y)\nonumber\\
            &\geq \lambda_{\mathbb{e}}(\mathbb{B}^{\mathbb{e}}_{\delta}(x)\backslash \mathbb{B}^{\mathbb{e}}_{\delta/2}(z_{0}))\ell+\lambda_{\mathbb{e}}(\mathbb{B}^{\mathbb{e}}_{\delta/2}(z_0)\backslash \mathbb{B}^{\mathbb{e}}_{\delta/4}(z_{1}))\left(\ell-\frac{1}{2}\delta\right)+\frac{1}{2}\delta\lambda_{\mathbb{e}}(\mathbb{O}_\delta^{\mathbb{e}}(x,y)\backslash\mathbb{O}_{\delta/2}^{\mathbb{e}}(x,y))\nonumber\\
            &\qquad +\frac{1}{4}\delta\lambda_{\mathbb{e}}(\mathbb{O}_{\delta/2}^{\mathbb{e}}(z_0,y)\backslash\mathbb{O}_{\delta/4}^{\mathbb{e}}(z_0,y))+\int_{\mathbb{B}^{\mathbb{e}}_{\delta/4}(z_{1})}\d\lambda_{\mathbb{e}}(\sigma)\rho_{\mathbb{e}}(\sigma,y)\nonumber\\
            &=\lambda_{\mathbb{e}}(\mathbb{B}^{\mathbb{e}}_{\delta}(x)\backslash \mathbb{B}^{\mathbb{e}}_{\delta/4}(z_{1}))\ell+\frac{1}{4}\delta\left(2\lambda_{\mathbb{e}}(\mathbb{O}_\delta^{\mathbb{e}}(x,y)\backslash\mathbb{O}_{\delta/2}^{\mathbb{e}}(x,y))+\lambda_{\mathbb{e}}(\mathbb{O}_{\delta/2}^{\mathbb{e}}(z_0,y)\backslash\mathbb{O}_{\delta/4}^{\mathbb{e}}(z_0,y))\right)\nonumber\\
            &\qquad -\frac{1}{2}\delta \lambda_{\mathbb{e}}(\mathbb{B}^{\mathbb{e}}_{\delta/2}(z_0)\backslash \mathbb{B}^{\mathbb{e}}_{\delta/4}(z_{1}))+\int_{\mathbb{B}^{\mathbb{e}}_{\delta/4}(z_{1})}\d\lambda_{\mathbb{e}}(\sigma)\rho_{\mathbb{e}}(\sigma,y).
        \end{align}
        Rather than further repeat this recursion we now find a different bound for the final integral term. Indeed, first note that for any $u\in [[x,y]]\cap \mathbb{B}_{\delta/4}^{\mathbb{e}}(z_{1})$ we have $\rho_{\mathbb{e}}(u,y)=\ell-\rho_{\mathbb{e}}(x,u)$ trivially. At the same time, for any $u\in \mathbb{B}_{\delta/4}^{\mathbb{e}}(z_{1})$ that is not colinear with $x$ and $y$ we have a $\sigma_u=\Lambda(xyu)\in [[x,y]]\cap \mathbb{B}_{\delta}^{\mathbb{e}}(x)$ such that 
        \begin{align}
            \rho_{\mathbb{e}}(u,y)=\rho_{\mathbb{e}}(u,\sigma_u)+\rho_{\mathbb{e}}(\sigma_u,y)=\ell-\rho_{\mathbb{e}}(x,\sigma_u)+\rho_{\mathbb{e}}(u,\sigma_u)\geq \ell-(\rho_{\mathbb{e}}(x,\sigma_u)+\rho_{\mathbb{e}}(u,\sigma_u))=\ell-\rho_{\mathbb{e}}(x,u).\nonumber
        \end{align}
        for all $u\in \mathbb{B}_{\delta/2}^{\mathbb{e}}(z_{\delta})$. Hence
        \begin{align}\label{inequality: MainTheorem2}
            \int_{\mathbb{B}^{\mathbb{e}}_{\delta/4}(z_{1})}\d\lambda_{\mathbb{e}}(\sigma)\rho_{\mathbb{e}}(\sigma,y)\geq \int_{\mathbb{B}^{\mathbb{e}}_{\delta/4}(z_{1})}\d\lambda_{\mathbb{e}}(\sigma)(\ell-\rho_{\mathbb{e}}(x,\sigma))\geq (\ell-\delta)\lambda_{\mathbb{e}}(\mathbb{B}^{\mathbb{e}}_{\delta/4}(z_{1}))\nonumber
        \end{align}
        where the final inequality follows from evaluating the second expression after recognising that $\mathbb{B}^{\mathbb{e}}_{\delta/4}(z_{1})\subseteq \mathbb{B}^{\mathbb{e}}_{\delta}(x)$. Hence combining the inequalities \ref{inequality: MainTheorem1} and \ref{inequality: MainTheorem2} we obtain
        \begin{align}
            \int_{\mathbb{B}_\delta^{\mathbb{e}}(x)}\d\lambda_{\mathbb{e}}(\sigma)\rho_{\mathbb{e}}(\sigma,y)&\geq \lambda_{\mathbb{e}}(\mathbb{B}^{\mathbb{e}}_{\delta}(x))\ell +\frac{1}{4}\delta\left(2\lambda_{\mathbb{e}}(\mathbb{O}_\delta^{\mathbb{e}}(x,y)\backslash\mathbb{O}_{\delta/2}^{\mathbb{e}}(x,y))+\lambda_{\mathbb{e}}(\mathbb{O}_{\delta/2}^{\mathbb{e}}(z_0,y)\backslash\mathbb{O}_{\delta/4}^{\mathbb{e}}(z_0,y))\right)\nonumber\\
            &\qquad -\frac{1}{2}\delta (\lambda_{\mathbb{e}}(\mathbb{B}^{\mathbb{e}}_{\delta/2}(z_0)\backslash \mathbb{B}^{\mathbb{e}}_{\delta/4}(z_{1}))+2 \lambda_{\mathbb{e}}(\mathbb{B}^{\mathbb{e}}_{\delta/4}(z_{1})))\nonumber.
        \end{align}
        Let $\alpha(\delta)$ and $f(\delta)$ be defined as in fact \ref{fact: ElementaryFact}. Recognising from proposition \ref{proposition: RerootingInvarianceVolume} and corollary \ref{corollary: Volumes} that
        \begin{align}
            \alpha(\delta)=\mathbb{E}(\lambda_{\mathbb{e}}(\mathbb{B}^{\mathbb{e}}_{\delta}(u)))=2\mathbb{E}(\lambda_{\mathbb{e}}(\mathbb{O}^{\mathbb{e}}_{\delta}(u,v))),\nonumber
        \end{align}
        we see that
        \begin{align}
            g(\delta)&\coloneqq \frac{1}{4}\mathbb{E}\left(\frac{2\lambda_{\mathbb{e}}(\mathbb{O}_\delta^{\mathbb{e}}(x,y)\backslash\mathbb{O}_{\delta/2}^{\mathbb{e}}(x,y))+\lambda_{\mathbb{e}}(\mathbb{O}_{\delta/2}^{\mathbb{e}}(z_0,y)\backslash\mathbb{O}_{\delta/4}^{\mathbb{e}}(z_0,y))-2\lambda_{\mathbb{e}}(\mathbb{B}^{\mathbb{e}}_{\delta/2}(z_0)\backslash \mathbb{B}^{\mathbb{e}}_{\delta/4}(z_{1}))-4 \lambda_{\mathbb{e}}(\mathbb{B}^{\mathbb{e}}_{\delta/4}(z_{1}))}{\lambda_{\mathbb{e}}(\mathbb{B}^{\mathbb{e}}_{\delta}(x))}\right)\nonumber\\
            &=\frac{1}{8}\frac{2\alpha(\delta)-2\alpha(\delta/2)+\alpha(\delta/2)-\alpha(\delta/4)-4\alpha(\delta/2)+4\alpha(\delta/4)-8\alpha(\delta/4)}{\alpha(\delta)}\nonumber\\
            &=\frac{1}{4}-\frac{\alpha(\delta/2)+9\alpha(\delta/4)}{8\alpha(\delta)}\nonumber\\
            &=f(\delta)\nonumber
        \end{align}
        and
        \begin{align}
            \mathbb{E}(\mathcal{K}_{\mathbb{e}}^{\delta,x,y}(f_{x,y}^{\mathbb{e}}))\geq \ell+\delta f(\delta)\nonumber.
        \end{align}
        Our claim thus relies on the behaviour of $f(\delta)$ for $\delta\downarrow 0$. In particular, for every $\varepsilon>0$ we have by fact \ref{fact: ElementaryFact} that 
        \begin{align}
            \mathbb{E}(\mathcal{K}_{\mathbb{e}}^{\delta,x,y}(f_{x,y}^{\mathbb{e}}))\geq \ell+\delta f(\delta)>\ell+\left(\frac{19}{128}-\varepsilon\right)\delta \nonumber
        \end{align}
        for all sufficiently small $\delta \in (0,\infty)$ as required.
    \end{proof}
    \begin{corollary}\label{corollary: OllCurvBound}
        Let $(\mathcal{T}_{\mathbb{e}},\rho_{\mathbb{e}})$ be the Brownian continuum random tree. For each $x\in \mathcal{T}_{\mathbb{e}}$ let $\gamma_x:[0,T]\rightarrow \mathcal{T}_{\mathbb{e}}$ denote a length-minimising geodesic in $\mathcal{T}_{\mathbb{e}}$ such that $\gamma_x(0)=x$. Then for $\lambda_{\mathbb{e}}$-almost all $x\in \mathcal{T}_{\mathbb{e}}$ and for every $\varepsilon>0$ we have
        \begin{align}
            -2\frac{\delta}{\ell}\leq \kappa^{\delta,\ell}_x(\gamma)<-\left(\frac{19}{128}-\varepsilon\right)\frac{\delta}{\ell}
        \end{align}
        for almost all $\ell\in (0,\infty)$ and for all sufficiently small $\delta>0$. In particular the scale-free Ollivier curvature $\kappa_x(\gamma)=-\infty$ for $\lambda_{\mathbb{e}}$-almost all $x\in \mathcal{T}_{\mathbb{e}}$ and $\mathcal{T}_{\mathbb{e}}$ as expected. $\qed$
    \end{corollary}
    \begin{remarks}
        The upper bound obtained for the Wasserstein distance above is in some sense the worst possible nontrivial bound since it is a uniform upper bound on the transport cost for all transport plans. No doubt it could be improved by looking more closely at the properties of particular transport plans. Similarly, we could also attempt to improve the lower bound on the Wasserstein distance e.g. by iterating the application of lemma \ref{lemma: RecursiveBound} further. There seems little advantage in doing this since the second term on the right-hand side of the inequality in lemma \ref{lemma: RecursiveBound} is obtained in a manner that seems impossible to make optimal. More precisely we have been rather cavalier in our estimate of the quantity
        \begin{align}
            \int_{\mathbb{O}_\delta^{\mathbb{e}}(x,y)}\d\lambda_{\mathbb{e}}(\sigma)\rho_{\mathbb{e}}(\sigma,x)\nonumber
        \end{align}
        which introduces a source of error independently of any error associated with the estimate of the third term 
        \begin{align}
            \int_{\mathbb{B}^{\mathbb{e}}_{\delta/2}(z_{\delta})}\d\lambda_{\mathbb{e}}(\sigma)\rho_{\mathbb{e}}(\sigma,y)\nonumber,
        \end{align}
        which can be reduced through iteration. Indeed this is a vital part of the proof: if we apply lemma \ref{lemma: RecursiveBound} only once and use the same procedure as in the proof to estimate this final integral term, the sign of the $\delta$ component is negative and the corresponding bound becomes trivial when considering the Ollivier curvature. Independently of these considerations the present bounds are sufficient to provide us with upper and lower bounds on the Ollivier curvature which differ only up to a choice of constant. The Ollivier curvature itself can converge to nontrivial finite negative values in the limits $\delta,\:\ell\rightarrow 0$ if we take e.g. $\delta=\ell/2$. However it is clear that the scale-free Ollivier curvature must diverge to \textit{negative} infinity as $\delta,\:\ell\rightarrow 0$ since up to constant we have upper and lower bounds given by $-(\delta\ell)^{-1}$. Note that $\ell=c\delta^{-1}$ for some constant $c>0$ is not a valid assignment since as $\delta\rightarrow 0$ this requires $\ell\rightarrow \infty$ while $\mathcal{T}_{\mathbb{e}}$ is compact. 
    \end{remarks}
    
    \appendix  
    \section{A Useful Fact}
    In this section we prove the following fact which we use in the computation of the Ollivier curvature bounds. While the fact is entirely elementary, even the slightly informal proof we present below is somewhat tedious to carry through. In practice one can easily verify the truth of this fact by plotting the requisite function using some suitable software package.
    \begin{fact}\label{fact: ElementaryFact}
        Define the functions
        \begin{align}
            \alpha(\delta)\coloneqq 1-\exp(-2\delta^2) && f(\delta)\coloneqq \frac{1}{4}-\frac{\alpha(\delta/2)+9\alpha(\delta/4)}{8\alpha(\delta)}
        \end{align}
        for $\delta\in \mathbb{R}$. $f(\delta)$ is continuously differentiable everywhere with
        \begin{align}
            f(0)=\frac{19}{128} && f'(0)=0.
        \end{align}
        In particular the turning point at $\delta=0$ is a maximum and for each $\varepsilon>0$ there is a neighbourhood $(-a,a)$ of $0$ in $\mathbb{R}$ such that
        \begin{align}
            \frac{19}{128}-\varepsilon<f(\delta)\leq \frac{19}{128}
        \end{align}
        for all $\delta\in (-a,a)$.
    \end{fact}
    \begin{proof}
        $\alpha(\delta)$ takes values in $[0,1)$ and vanishes uniquely at $\delta=0$; also it is manifestly smooth for all $\delta$. Thus we see that $f(\delta)$ is also manifestly smooth for all points in $\mathbb{R}$, with the possible exception of $\delta=0$. Noting that 
        \begin{align}
            \alpha'(a\delta)=4a^2\delta \exp(-2a^2\delta^2)=4a^2\delta(1-\alpha(a\delta))\nonumber
        \end{align}
        for all $a\in \mathbb{R}$ we can use the L'H\^{o}pital rule to evaluate
        \begin{align}
            f(0)=\frac{1}{4}-\lim_{\delta\rightarrow 0}\frac{\alpha'(\delta/2)+9\alpha'(\delta/4)}{8\alpha'(\delta)}=\frac{1}{4}-\lim_{\delta\rightarrow 0}\frac{4 \exp\left(-\frac{1}{2}\delta^2\right)+9\exp\left(-\frac{1}{8}\delta^2\right)}{128 \exp(-2\delta^2)}=\frac{19}{128}\nonumber.
        \end{align}
        Thus by continuity there is a neighbourhood of $0$ in which $f(\delta)>0$ for all values of $\delta$ in that neighbourhood. In fact $\delta=0$ is a local maximum. To see this first note that
        \begin{align}\label{equation: FactDerivativeSpecialFunction}
            f'(\delta)&=-\frac{8\alpha(\delta)(\alpha'(\delta/2)+9\alpha'(\delta/4))-8(\alpha(\delta/2)+9\alpha(\delta/4))\alpha'(\delta)}{64\alpha(\delta)^2}\nonumber\\
            &=\frac{\delta \exp(-2\delta^2)(\alpha(\delta/2)+9\alpha(\delta/4))}{2\alpha(\delta)^2}-\frac{\delta\left(4 \exp\left(-\frac{1}{2}\delta^2\right)+9\exp\left(-\frac{1}{8}\delta^2\right)\right)}{32\alpha (\delta)}.
        \end{align}
        Again this is manifestly smooth everywhere except perhaps at $\delta=0$ and we may use the L'H\^{o}pital rule to find
        \begin{align}
            f'(0)&=\lim_{\delta\rightarrow 0}\left(\frac{(1-4\delta^2)\exp(-2\delta^2)(\alpha(\delta/2)+9\alpha(\delta/4))+\delta\exp(-2\delta^2)(\alpha'(\delta/2)+9\alpha'(\delta/4))}{4\alpha(\delta)\alpha'(\delta)}\nonumber\right.\\
            &\qquad \left.-\frac{4\left(4 \exp\left(-\frac{1}{2}\delta^2\right)+9\exp\left(-\frac{1}{8}\delta^2\right)\right)-\delta^2\left(16 \exp\left(-\frac{1}{2}\delta^2\right)+9\exp\left(-\frac{1}{8}\delta^2\right)\right)}{128\alpha'(\delta)}\right)\nonumber\\
            &=\lim_{\delta\rightarrow 0}\frac{\delta\left(16 \exp\left(-\frac{1}{2}\delta^2\right)+9\exp\left(-\frac{1}{8}\delta^2\right)\right)}{512\exp(-2\delta^2)}-2\lim_{\delta\rightarrow 0}\frac{\delta (\alpha (\delta/2)+9\alpha(\delta/4))}{8\alpha(\delta)}\nonumber\\
            &\qquad +\lim_{\delta\rightarrow 0}\left(\frac{\delta\left(4\exp\left(-\frac{1}{2}\delta^2\right)+9\exp\left(-\frac{1}{8}\delta^2\right)\right)}{128\alpha(\delta)}+\frac{\alpha(\delta/2)+9\alpha(\delta/4)}{16\delta \alpha(\delta)}-\frac{4 \exp\left(-\frac{1}{2}\delta^2\right)+9\exp\left(-\frac{1}{8}\delta^2\right)}{128\delta \exp(-2\delta^2)}\right)\nonumber.
        \end{align}
        Let us denote the three limits above by $A$, $B$ and $C$ respectively. $A$ vanishes trivially; $B$ also vanishes if we notice that the quantity inside the limit is a product of two quantities ($\delta$ and $(\alpha(\delta/2)+9\alpha(\delta/4)/8\alpha(\delta))$) which have finite limits as $\delta\rightarrow 0$ (in fact $0$ and $13/128$ as shown above respectively). Thus we have
        \begin{align}
            f'(0)&=\lim_{\delta\rightarrow 0}\left(\frac{\delta\left(4\exp\left(-\frac{1}{2}\delta^2\right)+9\exp\left(-\frac{1}{8}\delta^2\right)\right)}{128\alpha(\delta)}+\frac{\alpha(\delta/2)+9\alpha(\delta/4)}{16\delta \alpha(\delta)}-\frac{4 \exp\left(-\frac{1}{2}\delta^2\right)+9\exp\left(-\frac{1}{8}\delta^2\right)}{128\delta \exp(-2\delta^2)}\right)\nonumber
        \end{align}
        Let us again denote the first, second and third terms inside the limits by $A$, $B$ and $C$ respectively. To evaluate the sum of these quantities in the limit $\delta\rightarrow 0$ it is helpful to use the approximations
        \begin{align}
            \exp(-2(a\delta)^2)=1+\mathcal{O}(\delta^2) && \alpha (a\delta)=2a^2\delta^2(1+\mathcal{O}(\delta^2))\nonumber
        \end{align}
        where $\mathcal{O}(\delta^n)$ indicates that the unwritten component is a convergent series with each term containing a factor $\delta^m$ for $m\geq n$; in particular we note that the limit of the series can be made arbitrarily close to $0$ as $\delta\rightarrow 0$. Thus we find that
        \begin{align}
            A=\frac{4\left(1+\mathcal{O}(\delta^2)\right)+9(1+\mathcal{O}(\delta^2))}{256 \delta}(1+\mathcal{O}(\delta^2))^{-1}=\frac{13}{256\delta}+\mathcal{O}(\delta).\nonumber
        \end{align}
        Similarly,
        \begin{align}
            B=\frac{4\delta^2(1+\mathcal{O}(\delta^2))+9\delta^2(1+\mathcal{O}(\delta^2))}{256\delta^3}
            (1+\mathcal{O}(\delta^2))^{-1}=\frac{13}{256 \delta}+\mathcal{O}(\delta)\nonumber.
        \end{align}
        Finally
        \begin{align}
            C=-\frac{4(1+\mathcal{O}(\delta^2))+9(1+\mathcal{O}(\delta^2))}{128\delta}(1+\mathcal{O}(\delta^2))^{-1}=-\frac{13}{128\delta}+\mathcal{O}(\delta)\nonumber.
        \end{align}
        Hence
        \begin{align}
            f'(0)=\lim_{\delta\rightarrow 0}\mathcal{O}(\delta)=0\nonumber
        \end{align}
        as required. It only remains to prove that the turning point at $\delta=0$ is a maximum since the rest of the statement follows immediately if this holds. Noting that $\exp(-a\delta^2)$ and $\alpha(\delta)$ are even functions we see that the turning point at $\delta=0$ is necessarily a maximum or a minimum. The former case suggests that $f'(\delta)<0$ for sufficiently small positive $\delta>0$. Define
        \begin{align}
            A\coloneqq \frac{(1-\alpha(\delta))(\alpha(\delta/2)+9\alpha(\delta/4))}{2\alpha(\delta)^2} && B\coloneqq \frac{4(1-\alpha(\delta/2))+9(1-\alpha(\delta/8))}{32\alpha(\delta)}\nonumber
        \end{align}
        Clearly $A,\:B>0$ for $\delta>0$. Also comparison with equation \ref{equation: FactDerivativeSpecialFunction} indicates that $f'(\delta)=\delta(A-B)$ so for $\delta>0$, $f'(\delta)<0$ iff $A<B$ iff $A/B<1$. But in particular
        \begin{align}
            \frac{A}{B}=\frac{16(1-\alpha(\delta))(\alpha(\delta/2)+9\alpha(\delta/4))}{\alpha(\delta)(4(1-\alpha(\delta/2))+9(1-\alpha(\delta/8)))}\nonumber.
        \end{align}
        Using the approximation 
        \begin{align}
            \alpha(\delta)=2\delta^2-2\delta^4+\mathcal{O}(\delta^6)\nonumber
        \end{align}
        we find that
        \begin{align}
            \frac{A}{B}&=\frac{16(1-2\delta^2+2\delta^4+\mathcal{O}(\delta^6))\left(16(4\delta^2-\delta^4+\mathcal{O}(\delta^6))+9(16\delta^2-\delta^4+\mathcal{O}(\delta^6))\right)}{(2\delta^2-2\delta^4+\mathcal{O}(\delta^6))(64(8-4\delta^2+\delta^4+\mathcal{O}(\delta^6))+9(128-16\delta^2+\delta^4+\mathcal{O}(\delta^6)))}\nonumber\\
            &=\frac{3328\delta^2-7056\delta^4+\mathcal{O}(\delta^6)}{3328\delta^2-2528\delta^4+\mathcal{O}(\delta^6)}\nonumber\\
            &=\left(1-\frac{441}{208}\delta^2+\mathcal{O}(\delta^4)\right)\left(1-\frac{79}{104}\delta^2+\mathcal{O}(\delta^4)\right)^{-1}\nonumber\\
            &=1-\frac{283}{208}\delta^2+\mathcal{O}(\delta^4)\nonumber
        \end{align}
        which is smaller than one as required for sufficiently small $\delta$. 
    \end{proof}
    \printbibliography
\end{document}